\documentclass[envcountsect,envcountsame]{svjour3}

\usepackage{amsmath,amsfonts}
\usepackage{interval}
\usepackage{algorithm}
\usepackage{algorithmicx}
\usepackage{algpseudocode}
\usepackage{tikz}
\usepackage{pgfplots}

\usepackage{mathtools}
\usepackage{todonotes}
\usepackage{amssymb}

\newcommand{\RR}{\mathbb{R}}
\newcommand{\NN}{\mathbb{N}}
\newcommand{\one}{\mathbf{1}}
\newcommand{\norm}[1]{\|#1\|}
\newcommand{\abs}[1]{|#1|}
\newcommand{\scp}[2]{\langle #1,#2\rangle}

\newcommand{\dd}{\mathrm{d}}
\newcommand{\bigO}{\mathcal{O}}

\DeclareMathOperator{\diag}{diag}

\newcommand{\calL}{\mathcal{L}}
\newcommand{\calF}{\mathcal{F}}
\newcommand{\frakM}{\mathfrak{M}}
\newcommand{\Kern}{\operatorname{kern}}
\newcommand{\weak}{\rightharpoonup}

\newtheorem{assumption}{Assumption}

\smartqed

\title{Quadratically regularized optimal transport}
\author{Dirk A. Lorenz \and Paul Manns \and Christian Meyer}
\institute{Dirk A. Lorenz, \email{d.lorenz@tu-braunschweig.de} \at TU Braunschweig, Institute of Analysis and Algebra, Germany\and Paul Manns, \email{p.manns@tu-braunschweig.de} \at TU Braunschweig, Institute of Mathematical Optimization,  Germany\and Christian Meyer, \email{christian.meyer@math.tu-dortmund.de} \at TU Dortmund, Fakult\"at f\"ur Mathematik, Germany}

\begin{document}

\maketitle
\begin{abstract}
  We investigate the problem of optimal transport in the so-called
  Kantorovich form, i.e.\ given two Radon measures on two compact sets,
  we seek an optimal transport plan which is another Radon measure on
  the product of the sets that has these two measures as marginals and
  minimizes a certain cost function.
  
  We consider quadratic regularization of the problem, which forces
  the optimal transport plan to be a square integrable function rather
  than a Radon measure. We derive the dual problem and show strong
  duality and existence of primal and dual solutions to the
  regularized problem. Then we derive two algorithms to solve the dual
  problem of the regularized problem: A Gauss-Seidel method and a
  semismooth quasi-Newton method and investigate both methods
  numerically. Our experiments show that the methods perform well even
  for small regularization parameters. Quadratic regularization is of
  interest since the resulting optimal transport plans are sparse,
  i.e.\ they have a small support (which is not the case for the often
  used entropic regularization where the optimal transport plan always
  has full measure).
\end{abstract}


\section{Introduction}
\label{sec:intro}

In this paper we will investigate a regularized version of the optimal
transport problem. Optimal transport dates back to the work of Monge
in 1781 but the problem formulation we use here is the one of
Kantorovich~\cite{kantorovich1942translocation}.  Let us fix some
notation and formulate the problem: Let
$\Omega_{1}\subset\RR^{d_{1}}$, $\Omega_{2}\subset \RR^{d_{2}}$ be two
compact domains, denote $\Omega=\Omega_{1}\times \Omega_{2}$, and
assume we are given two positive regular Radon measures $\mu_{1}$ and $\mu_{2}$ on
$\Omega_{1}$ and $\Omega_{2}$, respectively. Further we assume that a cost function
$c:\Omega_{1}\times\Omega_{2}\to \RR$ is given that models the cost of
transporting a unit of mass from $x_{1}\in\Omega_{1}$ to
$x_{2}\in\Omega_{2}$. The optimal transport problem asks to find a
transport plan $\pi$, which is a Radon measure on $\Omega$, such that
it has minimal overall transport cost
$\int_{\Omega}c(x_{1},x_{2})\,\dd\pi(x_{1},x_{2})$ among all measures
$\pi$ which have $\mu_{1}$ and $\mu_{2}$ as first and second
marginals, respectively, i.e.\ for all Borel sets $A\in\Omega_{1}$ it
holds that $\pi(A\times\Omega_{2}) = \mu_{1}(A)$ and for all Borel
sets $B\in\Omega_{2}$ it holds that
$\pi(\Omega_{1}\times B) = \mu_{2}(B)$.  This problem has been studied
extensively and we refer to the
books~\cite{rachev1998massI,rachev1998massII,villani2003topics,villani2008optimal,santambrogio2015optimal}. One particular result is, that an optimal plan $\pi^{*}$ exists and
that the support of optimal plans is contained in the so-called
$c$-superdifferential of a $c$-concave function~\cite[Theorem
1.13]{ambrosio2013user}. For many cost functions $c$, this means that
optimal transport plans are supported on small sets and that they are
in fact singular with respect to the Lebesgue measure on
$\Omega$. This makes the numerical treatment of optimal transport
problems difficult and one can employ regularization to obtain
approximately optimal plans $\pi$ that are functions on $\Omega$. The
regularization method that has got the most attention recently is
regularization with the negative entropy of $\pi$ and we refer
to~\cite{papadakis2014proximal,cuturi2016smoothed,carlier2017convergence}. Entropic
regularization has gotten popular in machine learning applications due
to the fact that it allows for the very simple Sinkhorn algorithm (in
the discrete case), see~\cite{cuturi2013sinkhorn,genevay2018learning}
and also~\cite{peyre2019computational} for a recent and thorough
review of the computational aspects of optimal transport.

Regularizations different from entropic regularization has been much
less studied. We are only aware of works in the discrete case,
e.g.\ \cite{blondel2018smoothOT,essid2018quadratically}. In this work we
will investigate the case where we regularize the problem in
$L^{2}(\Omega)$. The paper is organized as follows: In
Section~\ref{sec:qrot-cont-disc} we state the problem and analyze
existence and duality. It will turn out that existence of solutions of
the dual problem will be quite tricky to show, but we will show that
dual solutions exist in respective $L^{2}$ spaces and that a
straightforward optimality system characterizes primal-dual
optimality. In Section~\ref{sec:algorithms} we derive two different
algorithms for the discrete version of the quadratically regularized
optimal transport problem, and in Section~\ref{sec:examples} we
comment on a simple discretization scheme and report numerical
examples.

\paragraph{Notation.}
We will abbreviate $x_{+} = \max(x,0)$ (and will apply this also to functions and to measures where $_{+}$ will mean the positive part from the Hahn-Jordan decomposition).
By $C(\Omega)$ we denote that space of continuous functions on $\Omega$ (and we will always work on compact sets) equipped with the supremum norm $\norm{\cdot}_{\infty}$ and by $\frakM(\Omega)$ we denote the space of Radon measures on a compact domain and we use the norm $\norm{\mu}_{\frakM} = \sup\{\int f\, \dd\mu\mid f\in C(\Omega),\ |f|\leq 1\}$. The Lebesgue measure will be $\lambda$ (and we also use $\lambda_{1}$ and $\lambda_{2}$ to specify the Lebesgue measure on sets $\Omega_{1}$ and $\Omega_{2}$, respectively). For convenience, we use $\abs{\Omega}$ for the Lebesgue measure of the set $\Omega$. Furthermore, for a Radon measure $w \in \frakM$, we denote the absolutely and singular part 
arising from the Lebesgue decomposition with respect to the Lebesgue measure by $w_{ac}$ and $w_s$, i.e.\ they satisfy 
$w_{ac} \ll \lambda$ and $w_s \perp \lambda$. Duality pairings are denoted by $\langle \cdot, \cdot \rangle$. If
both arguments of the duality pairing are positive and the duality pairing
does not necessarily exist,
e.g.\ for $\psi \in \mathfrak{M}(\Omega)$ and $x \in L^2(\Omega)$,
we set $\langle \psi, x \rangle \coloneqq +\infty$.

\section{Quadratic regularization in the continuous case}
\label{sec:qrot-cont-disc}
For the quadratically regularized optimal transport problem we seek a
transport plan $\pi\in L^{2}(\Omega_{1}\times \Omega_{2})$ which for a given
cost function $c\in L^{2}(\Omega_{1}\times\Omega_{2})$, a
regularization parameter $\gamma>0$, and given functions $\mu_{i}\in L^{2}(\Omega_{i})$, $i=1,2$ solves
\begin{equation}\label{qrot-cont}
\begin{split}
  \min_{\pi}\ \scp{c}{\pi}_{L^{2}} + \tfrac\gamma2
  \norm{\pi}_{L^{2}}^{2}\quad \text{subject to} \quad
  \int_{\Omega_{2}}\pi(x_{1},x_{2}) \,\dd\lambda_{2} & = \mu_{1}(x_{1}),\\
  \int_{\Omega_{1}}\pi(x_{1},x_{2})\,\dd\lambda_{1} & = \mu_{2}(x_{2}),\\
  \pi(x_{1},x_{2})& \geq 0
\end{split}
\end{equation}
where the constraints are understood pointwise almost everywhere.

\subsection{Solutions of the primal problem}
\label{sec:primal-solutions}

It is straight forward to show, that optimal transport plans exist:
\begin{lemma}\label{lem:primalsol}
  Problem~\eqref{qrot-cont} has an optimal solution if and only if
  $\mu_{1}\in L^{2}(\Omega_{1})$, $\mu_{2}\in L^{2}(\Omega_{2})$,
  $\mu_{1},\mu_{2}\geq 0$ almost everywhere, and $\int_{\Omega_{1}}\mu_{1}(x_{1}) \,\dd\lambda_{1}= \int_{\Omega_{2}}\mu_{2}(x_{2}) \,\dd\lambda_{2}$.
\end{lemma}
\begin{proof}
  Assume that there is an optimal solution $\pi^{*}\in L^{2}(\Omega_{1}\times \Omega_{2})$. By Jensen's inequality we get
  \[
  \begin{split}
    \int_{\Omega_{1}}\mu_{1}^{2}(x_{1})\,\dd\lambda_{1} & = \int_{\Omega_{1}}\left(\int_{\Omega_{2}}\pi^{*}(x_{1},x_{2})\,\dd\lambda_{2}\right)^{2}\,\dd\lambda_{1}\\
    & \leq |\Omega_{2}|\iint_{\Omega_{1}\times \Omega_{2}}\pi^{*}(x_{1},x_{2})^{2}\,\dd\lambda_{1}\,\dd\lambda_{2} <\infty
  \end{split}
  \]
  which shows $\mu_{1}\in L^{2}(\Omega_{1})$. The argument for
  $\mu_{2}$ is similar. Non-negativity of $\mu_{1}$ and $\mu_{2}$ follows from
  non-negativity of $\pi^{*}$. Finally, by Fubini's theorem
  \begin{align*}
    \int_{\Omega_{1}}\mu_{1}(x_{1})\,\dd\lambda_{1} & = \iint_{\Omega_{1}\times \Omega_{2}}\pi^{*}(x_{1},x_{2})\,\dd\lambda_{1}\,\dd\lambda_{2}\\
    & = \int_{\Omega_{2}}\mu_{2}(x_{2})\,\dd\lambda_{2}
  \end{align*}
  
  Conversely, if $\mu_{1}\in L^{2}(\Omega_{1})$ and
  $\mu_{2}\in L^{2}(\Omega_{2})$ and $\mu_{1},\mu_{2}\geq 0$ we set $C := \int_{\Omega_{1}}\mu_{1}(x_{1})\,\dd\lambda_{1}= \int_{\Omega_{2}}\mu_{2}(x_{2})\,\dd\lambda_{2}$.
  Then $\pi(x_{1},x_{2}) = \tfrac1C\mu_{1}(x_{1})\mu_{2}(x_{2})$ is feasible for~\eqref{qrot-cont} and since the objective is continuous, 
  coercive, and strongly convex a (unique) minimizer exists.\qed
\end{proof}

\subsection{Dual problem and existence of dual solutions}
\label{sec:dual-problem}

In the following section, we apply the classical Lagrange duality to the linear-quadratic program \eqref{qrot-cont}.
To this end, let us define the Lagrangian associated with \eqref{qrot-cont}. 
In order to shorten the notation, we set 
\begin{equation*}
  \mu \coloneqq \gamma\,\mu_{1}\otimes\mu_{2}.
\end{equation*}
Furthermore, we define 
\begin{equation}\label{eq:intproj}
  P_1 : L^2(\Omega) \ni \pi \mapsto \int_{\Omega_2} \pi \,\dd \lambda_2 \in L^2(\Omega_1),\,
  P_2 : L^2(\Omega) \ni \pi \mapsto \int_{\Omega_1} \pi \,\dd \lambda_1 \in L^2(\Omega_2),
\end{equation}
and denote the the primal objective by
\begin{equation}\label{eq:primobj}
  E_\gamma: L^2(\Omega) \to \RR, \quad E_\gamma(\pi) \coloneqq \int_{\Omega} c\, \pi \,\dd \lambda
  + \frac{\gamma}{2}\, \|\pi\|_{L^2(\Omega)}^2.
\end{equation}
Then, the Lagrangian associated with \eqref{qrot-cont} is given by
\begin{equation*}
\begin{aligned}
 & \calL : L^2(\Omega)\times L^2(\Omega_1) \times L^2(\Omega_2) \times L^2(\Omega)
 \to \RR,\\
 & \calL(\pi, \alpha_1, \alpha_2, \varrho) \coloneqq 
 \begin{aligned}[t]
  E_\gamma(\pi)  & - \scp{\varrho}{ \pi}_{L^2(\Omega)} \\
  & + \scp{\alpha_{1}}{ P_1 \pi - \mu_1}_{L^2(\Omega_1)} + \scp{\alpha_{2}}{ P_2 \pi - \mu_2}_{L^2(\Omega_2)}. 
 \end{aligned}
\end{aligned}
\end{equation*}
Then, by standard arguments, the primal problem in \eqref{qrot-cont} is equivalent to
\begin{equation}\tag{PP}\label{eq:primal}
  \inf_{\pi \in L^2(\Omega)} \;
  \sup_{\substack{\alpha_1 \in L^2(\Omega_1),\; \alpha_2 \in L^2(\Omega_2)  \\ \varrho\in L^2(\Omega),\; \varrho \geq 0}}
 \calL(\pi, \alpha_1, \alpha_2, \varrho),
\end{equation}
while its (Lagrangian) dual is given by
\begin{equation}\tag{DP}\label{eq:dual}
 \sup_{\substack{\alpha_1 \in L^2(\Omega_1),\; \alpha_2 \in L^2(\Omega_2)  \\ \varrho\in L^2(\Omega),\; \varrho \geq 0}}\;
 \inf_{\pi \in L^2(\Omega)} 
 \calL(\pi, \alpha_1, \alpha_2, \varrho).
\end{equation}
The main part of the upcoming analysis is devoted to the existence of solutions to \eqref{eq:dual}. 
Once this is established, the necessary and sufficient optimality condition associated with \eqref{qrot-cont}
in form of the variational inequality will allow us to derive an optimality system that is also amenable for numerical computations.

To show existence for \eqref{eq:dual}, we first reformulate the dual problem.
Since $\calL$ is quadratic w.r.t.~$\pi$, the inner $\inf$-problem is solved by 
\begin{equation}\label{eq:innerinf}
 \pi = \frac{1}{\gamma} (\rho + \alpha_1 \oplus \alpha_2 - c),
\end{equation}
where the mapping $\oplus : L^2(\Omega_1) \times L^2(\Omega_2) \to L^2(\Omega)$ is defined via
\begin{equation}\label{eq:oplusptwise}
 (v_1 \oplus v_2)(x_1, x_2) \coloneqq v_1(x_1) + v_2(x_2)
\end{equation}
for almost all $(x_1, x_2)\in \Omega$ and all $v_i\in L^2(\Omega_i)$, $i=1,2$.

\begin{remark}\label{rem:oplus_adjoint}
  The map $\oplus$ is related to the adjoints of the projections $P_{1}$ and $P_{2}$ from~\eqref{eq:intproj} by $\alpha_{1}\oplus\alpha_{2} = P_{1}^{*}\alpha_{1} + P_{2}^{*}\alpha_{2}$.
\end{remark}

Inserting \eqref{eq:innerinf} into \eqref{eq:dual} yields
\begin{equation}\label{eq:supsup}
\begin{aligned}
   \sup_{\alpha_1 \in L^2(\Omega_1), \alpha_2 \in L^2(\Omega_2)} \,\sup_{\rho \geq 0}
   \Big(-\frac{1}{2\gamma} \int_{\Omega} (\rho
   + \alpha_1\oplus \alpha_2 -c)^2\,\dd \lambda & \\[-2ex]
   + \int_{\Omega_1}  \mu_1 \alpha_1 \,\dd \lambda_1 & + \int_{\Omega_2} \mu_2\alpha_2 \,\dd \lambda_2\Big)  
\end{aligned}
\end{equation}
Again, the inner optimization problem is quadratic w.r.t.~$\rho$ so that its solution is given by 
\begin{equation}\label{eq:rhoeq}
  \rho = -(\alpha_1 \oplus\alpha_2 - c)_- .
\end{equation}
Inserted in \eqref{eq:supsup}, this results in the following dual problem
   \begin{equation}\tag{D}
    \label{eq:qrot-dual-objective}
    \left.
    \begin{aligned}
      \min \quad &
      \begin{aligned}[t]
        \Phi(\alpha_{1},\alpha_{2}) 
        \coloneqq \tfrac12\norm{(\alpha_{1}\oplus\alpha_{2} - c)_{+}}_{L^2(\Omega)}^{2} & \\[-1ex]
        - \gamma \langle \alpha_{1}, \mu_{1}\rangle 
        & - \gamma \langle \alpha_{2}, \mu_{2}\rangle       
      \end{aligned}\\
      \text{s.t.} \quad & \alpha_i \in L^2(\Omega_i), \, i =1,2.
    \end{aligned}
    \quad \right\}
  \end{equation}
To prove existence of solutions for this problem, we need to require the following 

\begin{assumption}\label{assu:dual}
 The domains $\Omega_1$ and $\Omega_2$ are compact. 
 Moreover, the cost function $c$ is in $L^2(\Omega)$ and fulfills $c \geq \underline{c} > - \infty$. 
 Furthermore, the marginals $\mu_1$ and $\mu_2$ satisfy
 $\mu_i \in L^2(\Omega_i)$ and $\mu_i \geq \delta > 0$, $i = 1,2$.
 In addition we assume that $\int_{\Omega_1} \mu_2 \,\dd\lambda_1 = \int_{\Omega_1} \mu_2 \,\dd\lambda_1 = 1$.
\end{assumption}

\begin{remark}
 The last assumption on the normalization of the marginals is just to ease the subsequent analysis 
 and can be relaxed by $\int_{\Omega_1} \mu_2 \,\dd\lambda_1 = \int_{\Omega_1} \mu_2 \,\dd\lambda_1$, 
 which is needed anyway to ensure the existence of a solution to the primal problem, see Lemma~\ref{lem:primalsol}.
\end{remark}

\begin{remark}\label{rem:dual_nonuniqueness}
  Note that there is an obvious source of non-uniqueness for the dual problem~\eqref{eq:qrot-dual-objective}: We can add a constant to $\alpha_{1}$ and subtract it from $\alpha_{2}$ and this does not change the dual objective, i.e for any constant $C$ it holds that $\Phi(\alpha_{1}+C,\alpha_{2}-C) = \Phi(\alpha_{1},\alpha_{2})$. This non-uniqueness will not cause trouble in the proofs and when convenient, we remove it, e.g. by demanding that $\int_{\Omega_{2}}\alpha_{2}d\lambda_{2} = 0$.
\end{remark}


Observe that the objective $\Phi$ in \eqref{eq:qrot-dual-objective} 
is also well defined for functions in $\alpha_i \in L^1(\Omega_i)$ 
with  $(\alpha_1 \oplus \alpha_2 - c)_+ \in L^2(\Omega)$. This gives rise to the following auxiliary dual problem:
  \begin{equation}\tag{D'}
    \label{eq:dualmassprob}
    \left.
    \begin{aligned}
      \min \quad & \Phi(\alpha_1, \alpha_2)\\
      \text{s.t.} \quad & \alpha_i\in L^1(\Omega_i),\, i = 1,2, \quad (\alpha_1 \oplus \alpha_2 - c)_+ \in L^2(\Omega).
    \end{aligned}
    \quad \right\}
  \end{equation}
Our strategy to prove existence of solutions to \eqref{eq:qrot-dual-objective} is now as follows:
\begin{enumerate}
 \item First, we show that \eqref{eq:dualmassprob} admits a solution 
 $(\alpha_1^*, \alpha_2^*) \in L^1(\Omega_1) \times L^1(\Omega_2)$, see Proposition~\ref{prop:exdualaux}.
 \item Then, we prove that $\alpha_1^*$ and $\alpha_2^*$ possess higher regularity, namely that they are
 functions in $L^2(\Omega_i)$, $i=1,2$, cf.~Theorem~\ref{thm:minimizers_dualproblem}.
 \item Thus, $(\alpha_1^*, \alpha_2^*)$ is feasible for \eqref{eq:qrot-dual-objective} and, since the feasible set of 
 \eqref{eq:dualmassprob} contains the one of \eqref{eq:qrot-dual-objective}, while the objective of \eqref{eq:dualmassprob} 
 restricted to $L^2$-functions coincides with the objective in \eqref{eq:qrot-dual-objective}, this finally gives 
 that $(\alpha_1^*, \alpha_2^*)$ is indeed optimal for \eqref{eq:qrot-dual-objective}.
\end{enumerate}
The reason to consider \eqref{eq:dualmassprob} is essentially that the objective $\Phi$ is not coercive in $L^2(\Omega)$, 
but only in $L^1(\Omega)$ (at least w.r.t.~the negative part of $\alpha_i$). Therefore, we have to deal with weakly$^*$ 
converging sequences in the space of Radon measures within the proof of existence of solutions.
For this purpose, we need to extend the objective to a suitable set. To that end, let us define
\begin{equation}\label{eq:defG}
 G : L^2(\Omega) \ni w \mapsto \int_{\Omega} \tfrac{1}{2}\, w_+^{2} - w \mu\,\dd \lambda \in \RR.
\end{equation}    
Note that, thanks to $\int_{\Omega_1} \mu_2 \,\dd\lambda_1 = \int_{\Omega_1} \mu_2 \,\dd\lambda_1 = 1$, it holds
\begin{equation}\label{eq:PhiG}
 \Phi(\alpha_1, \alpha_2) = G(\alpha_1 \oplus \alpha_2 - c) - \int_\Omega \,c\,\mu\,\dd\lambda
 \quad \forall\, \alpha_i\in L^2(\Omega_i), \, i=1,2.
\end{equation}
Of course, $G$ is also well defined as a functional on the feasible set of \eqref{eq:dualmassprob} 
and we will denote this functional by the same symbol to ease notation.
In order to extend $G$ to the space of Radon measures, 
consider for a given measure $w\in \frakM(\Omega)$, the Hahn-Jordan decomposition $w = w_+ - w_-$
and assume that $w_+ \in L^2(\Omega)$.
Then, we set $G(w) = \int_\Omega\tfrac{1}{2}\, w_+^2 \,\dd\lambda - \int_\Omega \mu \,\dd w$. 
With a slight abuse of notation, we denote this mapping by $G$, too.
Furthermore, for $w_+ \in L^2(\Omega)$,
$- \int_\Omega w_+ \mu\,\dd \lambda$ is finite for $\mu \in L^2(\Omega)$
as in Assumption~\ref{assu:dual}.
Regarding, the negative part, we define
$\int_\Omega \mu\,\dd w_{-} \coloneqq \infty$, where this expression
is not properly defined, as $w_-$ and $\mu$ are both positive.
Combining this, we obtain that
$- \int_\Omega \mu\, \dd w \in \mathbb{R} \cup \{\infty\}$.

Note in this context that, if the singular part of $w$ (w.r.t.~the Lebesgue measure) vanishes, 
then also $w_+\in L^1(\Omega)$ and $w_+(x) = \max\{0,w(x)\}$ $\lambda$-a.e.\ in $\Omega$
so that both functionals coincide on $L^2(\Omega)$, which justifies this notation.
Furthermore, we also generalize the map $\oplus$ to the measure space by setting
\begin{equation*}
 \alpha_1 \oplus \alpha_2 \coloneqq \alpha_1 \otimes \lambda_2 + \lambda_1 \otimes \alpha_2,\quad 
 \alpha_i \in \frakM(\Omega_i),\, i=1,2.
\end{equation*}
Again, it is easily seen that, for $\alpha_i \in L^2(\Omega_i)$, $i=1,2$, this definition boils down 
to the one in \eqref{eq:oplusptwise}. Also Remark~\ref{rem:oplus_adjoint} applies in that we can express $\alpha_{1}\oplus\alpha_{2}$ in terms of the adjoints of $P_{1}$ and $P_{2}$ from~\eqref{eq:intproj} when defined appropriately.

The next lemma is rather obvious and covers the coercivity of $G$ in $L^1(\Omega)$ as indicated above.

\begin{lemma}\label{lem:wnbound}
  Let Assumption~\ref{assu:dual} hold and suppose that a sequence $\{w^n\} \subset L^2(\Omega)$ fulfills
  \begin{equation*}
    G(w^n) \leq C < \infty \quad \forall \, n \in \NN.
  \end{equation*}    
  Then, the sequences $\{w^n_+\}$ and $\{w^n_-\}$ are bounded in $L^2(\Omega)$ and $L^1(\Omega)$, respectively.
\end{lemma}

\begin{proof}
 We rewrite $G$ as $G(w) = \int_{\Omega} \tfrac{1}{2}\,w_{+}^{2} - w_{+}\mu \,\dd\lambda + \int_{\Omega}w_{-}\mu \,\dd\lambda$.
 The positivity of $\mu$ then implies
 \begin{equation*}
  \|w^n_+\|_{L^2(\Omega)}^2 
  = G(w^n) +  \int_\Omega w^n_{+}\mu \,\dd\lambda - \int_{\Omega}w^n_{-}\mu \,\dd\lambda
  \leq C + \|\mu\|_{L^2(\Omega)}\, \|w^n_+\|_{L^2(\Omega)},
 \end{equation*}    
 which gives the first assertion. To see the second one, we use $\mu \geq \delta$ to estimate
 \begin{align*}
    C \geq G(w^{n}) 
    & = \int_{\Omega}\tfrac{1}{2}\,(w^{n}_{+}-\mu)^{2}\,\dd\lambda - \int_{\Omega}\mu^{2}/2 \,\dd\lambda +\int_{\Omega}w^{n}_{-}\mu\,\dd\lambda\\
    & \geq  - \int_{\Omega}\mu^{2}/2\,\dd\lambda +\delta\norm{w^{n}_{-}}_{L^1(\Omega)}, 
 \end{align*}
 which finishes the proof.\qed
\end{proof}

The next lemma provides a lower semicontinuity result for $G$ w.r.t.~weak$^*$ convergence in $\frakM(\Omega)$. 
Note that, here, we need the extension of $G$ as introduced above.

\begin{lemma}\label{lem:weaklimbound}
 Let Assumption~\ref{assu:dual} be fulfilled and a sequence $\{w_n\}\subset L^2(\Omega)$ be given such that 
 $w^n \weak^* w^*$ in $\frakM(\Omega)$ and $G(w^n)\leq C < \infty$ for all  $n\in \NN$. 
 Then there holds $w^*_+ \in L^2(\Omega)$ and 
 \begin{equation}\label{eq:lsc}
  G(w^*) \leq \liminf_{n\to \infty} G(w^n).
 \end{equation}
\end{lemma}

\begin{proof}
	By virtue of Lemma~\ref{lem:wnbound}, $\{w^n_+\}$ is bounded in $L^2(\Omega)$ and
	thus,
    there is a subsequence of $\{w^n_+\}$, to ease notation denoted by the same symbol, that converges weakly in $L^2(\Omega)$ 
    to some $\theta_+ \in L^2(\Omega)$. Since the set $\{v\in L^2(\Omega) : v\geq 0 \text{ a.e.\ in }\Omega\}$ is clearly weakly closed, 
    we have $\theta_+ \geq 0$ a.e.\ in $\Omega$. With a little abuse of notation, we denote the Radon measure induced by
    $C(\Omega) \ni \varphi \mapsto \int_\Omega \theta_+\,\varphi\,\dd\lambda \in \RR$
    by $\theta_+$, too.
    If we define $\theta_- := \theta_+ - w^* \in \frakM(\Omega)$, then $w^n_- = w^n_+ - w^n \weak^* \theta_-$ in $\frakM(\Omega)$ 
    with $\theta_- \geq 0$.
    Thus we have $w^* = \theta_+ - \theta_-$ with two positive Radon measures $\theta_+$, $\theta_-$.  The maximality property of 
    the Hahn-Jordan decomposition then implies $w^*_+ \leq \theta_+$. Since $\theta_+$ is absolutely continuous w.r.t.\ $\lambda$, 
    the same thus holds for $w^*_+$, i.e.\ $w^*_+ \in L^1(\Omega)$. Applying again $w^*_+ \leq \theta_+$, which clearly also holds for the densities 
    pointwise $\lambda$-almost everywhere, 
    we moreover deduce from the weak convergence of $w^n_+$ in $L^2(\Omega)$ that
    \begin{equation}\label{eq:lsc2}
        \int_\Omega (w^*_+)^2 \dd\lambda\leq \int_\Omega (\theta_+)^2 \dd\lambda 
        \leq \liminf_{n\to\infty} \int_\Omega (w^n_+)^2 \dd\lambda,
    \end{equation}
    which implies $w^*_+ \in L^2(\Omega)$ as claimed. 
    Since the above reasoning applies to every subsequence $w^n_+$  that is weakly converging in $L^2(\Omega)$, 
    \eqref{eq:lsc2} holds for the whole sequence $\{w^n_+\}$, which together with the weak$^*$ convergence of $w^n$ and the definition of $G$,
    gives \eqref{eq:lsc}.

\qed
\end{proof}

Before we are in the position to prove existence for \eqref{eq:dualmassprob}, 
we need two additional results on the $\oplus$-operator in the space of Radon measures.

\begin{lemma}\label{lem:bounded-min-seq-alpha12}
  If $\alpha_{i}\in \frakM(\Omega_{i})$, $i=1,2$ and $\int_{\Omega_{2}}d\alpha_{2} = 0$, then it holds that
  \[
  \norm{\alpha_{1}}_{\frakM} \leq \tfrac{1}{|\Omega_{2}|}\norm{\alpha_{1}\oplus\alpha_{2}}_{\frakM}\quad\text{and}\quad \norm{\alpha_{2}}_{\frakM} \leq \tfrac{2}{|\Omega_{1}|}\norm{\alpha_{1}\oplus\alpha_{2}}_{\frakM}
  \]
\end{lemma}
\begin{proof}
  We estimate 
  \begin{align}
    \norm{\alpha_{1}\oplus\alpha_{2}}_{\frakM} & = \sup_{\norm{\phi}_{\infty}\leq 1}\iint_{\Omega_{1}\times\Omega_{2}}\phi(x_{1},x_{2}) \,\dd(\alpha_{1}(x_{1}) + \alpha_{2}(x_{2}))\nonumber\\
    & \geq \sup_{\substack{\norm{\phi_{1}}_{\infty}\leq 1\\\norm{\phi_{2}}_{\infty}\leq 1}}\iint_{\Omega_{1}\times\Omega_{2}}\phi_{1}(x_{1})\phi_{2}(x_{2})\,\dd(\alpha_{1}(x_{1}) + \alpha_{2}(x_{2}))\nonumber\\
    & = \sup_{\substack{\norm{\phi_{1}}_{\infty}\leq 1\\\norm{\phi_{2}}_{\infty}\leq 1}}\Bigg[\iint_{\Omega_{1}\times\Omega_{2}}\phi_{1}(x_{1})\phi_{2}(x_{2})\,\dd\alpha_{1}(x_{1})\dd\lambda_{2}\nonumber\\
    & \qquad\qquad + \iint_{\Omega_{1}\times\Omega_{2}}\phi_{1}(x_{1})\phi_{2}(x_{2})\,\dd\lambda_{1}\dd\alpha_{2}(x_{2})\Bigg].\label{eq:aux-est-alpha12}
  \end{align}
  Taking $\phi_{2}\equiv 1$ and using $\int_{\Omega_{2}}\,\dd\alpha_{2}(x_{2}) = 0$ gives
  \begin{align*}
     \norm{\alpha_{1}\oplus\alpha_{2}}_{\frakM}  &\geq \sup_{\norm{\phi_{1}}_{\infty}\leq 1}\int_{\Omega_{1}}\phi_{1}(x_{1})d\alpha_{1}(x_{1})|\Omega_{2}| + \int_{\Omega_{2}}\dd\alpha_{2}(x_{2}) \int_{\Omega_{1}}\phi_{1}(x_{1})\dd\lambda_{1}\\
    & = |\Omega_{2}|\norm{\alpha_{1}}_{\frakM}.
  \end{align*}
  Now we start again at~\eqref{eq:aux-est-alpha12} and estimate from below by taking $\phi_{1}\equiv 1$ to get
  \begin{align*}
    \norm{\alpha_{1}\oplus\alpha_{2}}_{\frakM} & \geq \sup_{\norm{\phi_{2}}_{\infty}\leq 1}\int_{\Omega_{1}} \dd\alpha_{1}(x_{1})\int_{\Omega_{2}}\phi_2(x_{2})\,\dd\lambda_{2} + \int_{\Omega_{2}}\phi_{2}(x_{2})\,\dd\alpha(x_{2})|\Omega_{1}|\\
    & \geq -|\Omega_{2}|\int_{\Omega_{1}}d\alpha_{1}(x_{1}) + |\Omega_{1}|\norm{\alpha_{2}}_{\frakM}
  \end{align*}
  which implies
  \[
  |\Omega_{1}|\norm{\alpha_{2}}_{\frakM} \leq  \norm{\alpha_{1}\oplus\alpha_{2}}_{\frakM} + |\Omega_{2}|\norm{\alpha_{1}}_{\frakM}
  \]
  which completes the proof.\qed
\end{proof}

The next lemma will be used to show that the negative part of the minimizer of~\eqref{eq:qrot-dual-objective} does not have a singular part. 

\begin{lemma}\label{lem:positivepart_characterization}
Let $c \in L^1(\Omega)$ and
$\alpha_i \in \frakM(\Omega_i)$ for $i \in \{1,2\}$
with Lebesgue decompositions, $\alpha_i = f_i + \eta_i$ satisfying $f_i \ll \lambda$ and $\eta_i \perp \lambda$
for $i \in \{1,2\}$.
\begin{enumerate}
\item It holds that
  \begin{align}
    (\alpha_1 \oplus \alpha_2  - c)_+ &= (f_1 \oplus f_2 - c + (\eta_1)_+ \oplus (\eta_2)_+)_+. \label{eq:a1plusa2minusc_positive_part_characterization}
  \end{align}
\item If $(\alpha_{i})_{+}$ is absolutely continuous for $i=1,2$, then
  for $\tilde\alpha_{i} = \alpha_{i} - (\eta_{i})_{-}$
  for $i=1,2$, it holds that	
  \[
  \Phi(\tilde\alpha_{1},\tilde\alpha_{2})\leq\Phi(\alpha_{1},\alpha_{2}).\]
\end{enumerate}
\end{lemma}
\begin{proof}
We first proof point 1.
The measures $f_i$, $\eta_i$ exist by Lebesgue's 
decomposition theorem, see Theorem 1.155 in 
\cite{fonseca2007calculusofvariations}.
 We combine these decompositions with 
$\alpha_1 \oplus \alpha_2 = \alpha_1 \otimes \lambda + \lambda \otimes \alpha_2$ 
to arrive at Lebesgue's decomposition of $\alpha_1 \oplus \alpha_2$ with respect to $\lambda\otimes\lambda$, namely
\begin{align}
	\alpha_1 \oplus \alpha_2  - c &= f_1 \oplus f_2 - c + \eta_1 \oplus \eta_2 \label{eq:a1plusa2minusc_lebesgue_decomposition}\\
	f_1 \oplus f_2 - c &\ll \lambda \otimes \lambda \label{eq:f1plusf2minusc_ac_lebesgue} \\
	\eta_1 \oplus \eta_2 &\perp \lambda \otimes \lambda \label{eq:eta1pluseta2_perp_lebesgue} 
\end{align}
(which holds true because $c \in L^1(\Omega) \hookrightarrow \frakM(\Omega)$). 
Now, we consider the Hahn-Jordan decomposition of $\eta_1$,
\begin{align}
	\eta_1 &= (\eta_1)_+ - (\eta_1)_- \nonumber \\
	(\eta_1)_+ &\perp (\eta_1)_-, \label{eq:eta1p_perp_eta1m}
\end{align}
and obtain from~\eqref{eq:a1plusa2minusc_lebesgue_decomposition} that
\begin{align*}\label{eq:a1plusa2minusc_combined_decomposition}
	\alpha_1 \oplus \alpha_2  - c & = (f_{1} + \eta_{1})\oplus(f_{2}+\eta_{2}) - c\\
	 & = f_{1}\oplus f_{2}  + \eta_{1}\oplus\eta_{2}- c\\
	 & = f_{1}\oplus f_{2} + \big((\eta_{1})_{+}-(\eta_{1})_{-}\big)\oplus\eta_{2}- c\\
	 & = f_{1}\oplus f_{2} + (\eta_{1})_{+}\otimes\lambda -(\eta_{1})_{-}\otimes\lambda + \lambda\otimes\eta_{2}- c\\
   & = f_1 \oplus f_2 - c + (\eta_1)_+ \oplus \eta_2 - (\eta_1)_- \otimes \lambda.
\end{align*} 
Furthermore,
\[ (\eta_1)_- \otimes \lambda \perp f_1 \oplus f_2 - c + (\eta_1)_+ \oplus \eta_2 \]
where the singularity with respect to $f_1 \oplus f_2 - c$ is due to \eqref{eq:f1plusf2minusc_ac_lebesgue} and \eqref{eq:eta1pluseta2_perp_lebesgue} 
and the singularity with respect to $(\eta_1)_+ \oplus \eta_2$ is due to \eqref{eq:eta1p_perp_eta1m}. Thus, 
\begin{align*}
	(\alpha_1 \oplus \alpha_2  - c)_- &= (f_1 \oplus f_2 - c + (\eta_1)_+ \oplus \eta_2)_- + (-(\eta_1)_- \otimes \lambda)_- \\
	                                  &= (f_1 \oplus f_2 - c + (\eta_1)_+ \oplus \eta_2)_- + (\eta_1)_- \otimes \lambda. 
\end{align*}
as $(\eta_1)_- \otimes \lambda$ is a positive measure. Consequently,
\begin{align*}
	(\alpha_1 \oplus \alpha_2  - c)_+ &= (f_1 \oplus f_2 - c + (\eta_1)_+ \oplus \eta_2)_+.
\end{align*}
Repeating this argument with the Hahn-Jordan decomposition of $\eta_2$ yields the claim.

The second part of the lemma is a direct consequence of the first:
Since $(\alpha_{1}\oplus\alpha_{2}-c)_{+} = (\tilde\alpha_{1}\oplus\tilde\alpha_{2}-c)_{+}$,
the first summand in the functional $\Phi$ is equal for $\alpha_{i}$ and $\tilde\alpha_{i}$. However, the second summand in $\Phi$ can not decrease since $\tilde\alpha_{i}\leq\alpha_{i}$, $\mu_{i}\geq 0$ and
$\gamma \langle (\eta_{i})_{-}, \mu_{i} \rangle = \infty$
if the duality pairing does not exist.
\qed
\end{proof}

Now we are ready to prove the existence result for \eqref{eq:dualmassprob}:

\begin{proposition}\label{prop:exdualaux}
  Under Assumption~\ref{assu:dual} the minimization problem~\eqref{eq:dualmassprob} admits a
  solution $(\alpha_{1}^*,\alpha_{2}^*) \in L^1(\Omega_1) \times L^1(\Omega_2)$.
\end{proposition}

\begin{proof}
  We proceed via the classical direct method of the calculus of variations. 
  For this purpose, let $\{(\alpha_{1}^{n},\alpha_{2}^{n})\} \subset L^1(\Omega_1) \times L^1(\Omega_2)$ 
  with $(\alpha_1^n \oplus \alpha_2^n - c)_+ \in L^2(\Omega)$ 
  be a minimizing sequence for \eqref{eq:dualmassprob}, where we shift $\alpha_{1}$ and $\alpha_{2}$ 
  by adding and subtracting constants such that we obtain $\int_{\Omega_{2}}\alpha_2 \,\dd\lambda_2=0$. 
  Note that, due to its additive structure, this does not change the objective $\Phi$ in \eqref{eq:dualmassprob}, cf. Remark~\ref{rem:dual_nonuniqueness}.
  
  Next, let us define $w^{n} \coloneqq \alpha_{1}^{n}\oplus\alpha_{2}^{n}-c$. Then, thanks to 
  \eqref{eq:PhiG} and Lemma~\ref{lem:wnbound}, the sequence $\{w^n\}$ is bounded in $L^1(\Omega)$.
  Hence, there is a weakly$^*$ converging subsequence, which we denote by the same 
  symbol w.l.o.g., i.e.\ $w^n \weak^* \tilde w$ in $\frakM(\Omega)$. Now, Lemma~\ref{lem:weaklimbound} applies 
  giving that 
  \begin{align}
    & \tilde w_+ \in L^2(\Omega), \label{eq:wastfeas}\\
    & G(\tilde w) \leq \liminf_{n\to \infty} G(w^n). \label{eq:wastopt}
  \end{align}
  
  Since $\{w^n\}$ is bounded in $\frakM(\Omega)$, the same holds for $\{\alpha_1^n \oplus \alpha_2^n\}$ and, 
  as $\alpha_2^n$ is normalized, Lemma~\ref{lem:bounded-min-seq-alpha12} gives that $\{\alpha_i^n\}$ is bounded in 
  $\frakM(\Omega_i)$, $i=1,2$. Therefore, we can select a further (sub-)subsequence, still denoted by the same symbol to ease notation, 
  such that 
  \begin{equation*}
   \alpha_i^n \weak^* \tilde \alpha_i \quad \text{in } \frakM(\Omega_i), \quad i = 1,2.
  \end{equation*}
  Since the mapping $\frakM(\Omega_1) \times \frakM(\Omega_2) \ni (\alpha_1, \alpha_2) \mapsto \alpha_1 \oplus \alpha_2 \in \frakM(\Omega)$
  is the adjoint of the projection mapping $C(\Omega) \ni \varphi \mapsto \big(\int_{\Omega_2} \varphi \,\dd \lambda_2, \int_{\Omega_1} \varphi \,\dd \lambda_1\big)
  \in C(\Omega_1) \times C(\Omega_2)$, 
  see Remark~\ref{rem:oplus_adjoint}, it is weakly$^*$ continuous so that 
  \begin{equation}
   \tilde w = \tilde \alpha_{1}\oplus\tilde \alpha_{2}-c.
  \end{equation}
 
 Next, we investigate the singular parts of $\tilde \alpha_1$ and $\tilde \alpha_2$. We start with the positive part and 
 employ Lebesgue's decomposition of $\tilde \alpha_1$ and $\tilde \alpha_2$: 
 \begin{equation*}
  \tilde\alpha_i = \alpha_i^* + \tilde\eta_i, \quad
  \alpha_i^* \ll \lambda_i, \quad \tilde\eta_i \perp \lambda_i,\quad   i =1, 2.
 \end{equation*}
 In the following we will see that the regular parts $\alpha_i^* \in L^1(\Omega_i)$, $i=1,2$, are exactly the solution of 
 \eqref{eq:dualmassprob}. For this purpose, we first show that the positive parts of $\tilde\eta_1$ and $\tilde\eta_2$ vanish. 
 We have $\alpha_1^*\oplus \alpha_2^* - c \ll \lambda$, $\tilde\eta_1 \oplus \tilde \eta_2 \perp \lambda$, and, 
 by uniqueness of Lebesgue's decomposition,
 $\tilde w_s = \tilde \eta_1 \oplus \tilde \eta_2$. But from \eqref{eq:wastfeas}, we know that $(\tilde w_s)_+ = 0$. 
 Combining this fact with Lemma~\ref{lem:positivepart_characterization}, applied to the case $f_1 = 0$, $f_2 = 0$, and $c = 0$, 
 we obtain
 \[ (\tilde \eta_1 \oplus \tilde \eta_2)_+ = (\tilde \eta_1)_+ \oplus (\tilde \eta_2)_+. \]
 and consequently, $(\tilde \eta_i)_+ = 0$ for $i =1,2$ by positivity.  
 Therefore, $(\tilde \alpha_{i})_{+}$ are $L^{1}$-functions rather than measures
 and
\begin{align}\label{eq:alphaastfeas}
	\tilde{w}_+
	= (\tilde{\alpha}_1 \oplus \tilde{\alpha}_2 - c)_+
	= (\alpha_1^* \oplus \alpha_2^* - c)_+
\end{align}
 Now Lemma~\ref{lem:positivepart_characterization} shows feasibility of $(\alpha_1^*, \alpha_2^*)$ for \eqref{eq:dualmassprob} and we also see that
\begin{equation}\label{eq:dualopt}
\begin{aligned}
 \Phi(\alpha_1^*, \alpha_2^*)
  &= \tfrac{1}{2}\int_\Omega (\alpha_1^* \oplus \alpha_2^* - c)_+^2\,\dd\lambda 
  - \gamma \int_{\Omega_1} \mu_1\, \alpha_1^* \,\dd\lambda_1
  - \gamma \int_{\Omega_2} \mu_2\, \alpha_2^* \,\dd\lambda_2 \\
  &\leq G(\tilde \alpha_1 \oplus \tilde \alpha_2 - c) - \int_\Omega c\,\mu\,\dd\lambda\\
  &= G(\tilde w) - \int_\Omega c\,\mu\,\dd\lambda\\
  & \leq \liminf_{n\to\infty} \Phi(\alpha_1^n, \alpha_2^n),
\end{aligned}
\end{equation}
which demonstrates the optimality of $(\alpha_1^*, \alpha_2^*)$.
\qed
\end{proof}

In the following, we assume that $\int_{\Omega_2} \alpha_2^*\,\dd\lambda_2 = 0$. If this is not the case, 
then we can again shift $\alpha_1^*$ and $\alpha_2^*$ without changing the value of $\Phi$, cf.~Remark~\ref{rem:dual_nonuniqueness}.

\begin{theorem}\label{thm:minimizers_dualproblem}
 Let Assumption~\ref{assu:dual} hold.
 Then every optimal dual solution $(\alpha_1^*, \alpha_2^*)$ from Proposition~\ref{prop:exdualaux} satisfies 
 $\alpha_i^*\in L^2(\Omega_i)$, $i=1,2$, and is therefore also a solution of the original dual problem \eqref{eq:qrot-dual-objective}. Moreover, the negative parts of $\alpha_{i}^{*}$ are bounded and the function $\frac{1}{\gamma}\left(\left(\alpha_{1}^{*} \oplus
      \alpha_{2}^{*}\right)-c\right)_{+}$
  has marginals the $\mu_{1}$ and $\mu_{2}$.
\end{theorem}

\begin{proof}
 We again consider the positive and the negative part separately and start with $(\alpha_1^*)_-$. 
 Let $\varphi \in C_c^\infty(\Omega_1)$ and $t> 0$ be fixed, but arbitrary. Then, thanks to 
 \begin{equation*}
  0 \leq ((\alpha_1^* + t\, \varphi) \oplus \alpha_2^*-c)_+
  \leq (\alpha_1^* \oplus \alpha_2^*-c)_+ + t \, \varphi_+,
 \end{equation*}
 Proposition~\ref{prop:exdualaux} implies that $((\alpha_1^* + t \varphi) \oplus \alpha_2^*-c)_+ \in L^2(\Omega)$ so that 
 $(\alpha_1^* + t \varphi, \alpha_2^*)$ is feasible for \eqref{eq:dualmassprob}. Therefore, the optimality of $(\alpha_1^*, \alpha_2^*)$
 for \eqref{eq:dualmassprob} yields
 \begin{equation*}
  \tfrac{1}{2}
  \int_\Omega \frac{1}{t} \Big( ((\alpha_1^* + t\, \varphi) \oplus \alpha_2^*-c)_+^2 - (\alpha_1^* \oplus \alpha_2^*-c)_+^2\Big) \,\dd \lambda
  - \gamma\int_{\Omega_1} \mu_1 \,\varphi\,\dd\lambda_1 \geq 0 \quad \forall\, t > 0.
 \end{equation*}
 Owing to the continuous differentiability of $\RR \ni r \mapsto r_+^2 \in \RR$, the first integrand converges to 
 $2 (\alpha_1^* \oplus \alpha_2^*-c)_+\varphi$ $\lambda$-a.e.~in $\Omega$ for $t\searrow 0$. 
 Moreover, the Lipschitz continuity of the $\max$-function gives that 
 \begin{equation*}
  \frac{1}{t} \Big( ((\alpha_1^* + t\, \varphi) \oplus \alpha_2^*-c)_+^2 - (\alpha_1^* \oplus \alpha_2^*-c)_+^2\Big) 
  \leq |\varphi|^2 + 2\,|\varphi|\, (\alpha_1^* \oplus \alpha_2^*-c)_+^2 \quad \text{a.e.~in }\Omega
 \end{equation*}
 holds for $0<t\leq 1$.
 Hence, due to Lebesgue's dominated convergence theorem, we are allowed to pass to the limit $t\searrow 0$ and obtain in this way
 \begin{equation*}
  \int_{\Omega_1} \Big(\int_{\Omega_2}  (\alpha_1^* \oplus \alpha_2^*-c)_+ \,\dd\lambda_2 - \gamma\mu_1\Big)\varphi \,\dd\lambda_1
  \geq 0.
 \end{equation*}
 Since $\varphi \in C_c^\infty(\Omega)$ was arbitrary, the fundamental lemma of the calculus of variations thus gives
 \begin{equation}\label{eq:noc1}
  \int_{\Omega_2}  (\alpha_1^* \oplus \alpha_2^*-c)_+ \,\dd\lambda_2 = \gamma\mu_1\quad 
  \text{$\lambda_1$-a.e.\ in } \Omega_1.
 \end{equation}
 Next, define the following sequence of functions in $L^1(\Omega_2)$:
 \begin{equation*}
  f_n(x_2) \coloneqq (-n + \alpha_2^*(x_2) - \underline{c})_+, \quad n\in \NN,
 \end{equation*}
 where $\underline{c}$ is the lower bound for $c$ from Assumption~\ref{assu:dual}. Then we have $f_n \geq 0$
 $\lambda_2$-a.e.~$\Omega_2$ and $f_n \searrow 0$ $\lambda_2$-a.e.\ in $\Omega_2$ so that the monotone 
 convergence theorem gives
 \begin{equation*}
  \int_{\Omega_2} (-n + \alpha_2^*(x_2) - \underline{c})_+ \,\dd\lambda_2 = \int_{\Omega_2} f_n(x_2) \,\dd\lambda_2 \to 0
  \quad \text{as } n\to \infty.
 \end{equation*}
 Thus there exists $N\in \NN$ such that 
 \begin{equation}
  \int_{\Omega_2} (-N + \alpha_2^*(x_2) - \underline{c})_+ \,\dd\lambda_2 < \gamma\,\delta, 
 \end{equation}
 where $\delta>0$ is the threshold for $\mu_1$ from Assumption~\ref{assu:dual}. 
 Now assume that $\alpha_1^* \leq -N$ $\lambda_1$-a.e.\ on a set of $E \subset \Omega_1$ of positive Lebesgue measure.  
 Then 
 \begin{equation*}
 \begin{aligned}
  \int_{\Omega_2}  (\alpha_1^* \oplus \alpha_2^*-c)_+ \,\dd\lambda_2 
  & \leq \int_{\Omega_2}  (-N \oplus \alpha_2^*- \underline{c})_+ \,\dd\lambda_2 
  < \gamma\,\delta \leq \gamma\,\mu_1 \quad \text{$\lambda_1$-a.e.\ in } E,
 \end{aligned}
 \end{equation*}
 which contradicts \eqref{eq:noc1}. Therefore, $\alpha_1^* > -N$ $\lambda_1$-a.e.\ in $\Omega_1$, 
 which even implies that $(\alpha_1^*)_- \in L^\infty(\Omega_1)$. Concerning $(\alpha_2^*)_-$, one 
 can argue in exactly the same way to conclude that $(\alpha_2^*)_- \in L^\infty(\Omega_2)$, too.

 For the positive parts we find
 \begin{equation*}
 \begin{aligned}
  & |\Omega_2|\,\|\alpha_1^*\|_{L^2(\Omega_1)}^2 
  + |\Omega_1|\,\|\alpha_2^*\|_{L^2(\Omega_2)}^2\\
  &\quad = \int_\Omega |\alpha_1^* \oplus \alpha_2^*|^2 \,\dd\lambda  \qquad\qquad\qquad\qquad \big(\text{since } 
  \int_{\Omega_2}\alpha_2^*\,\dd\lambda_2 = 0\big)\\
  &\quad = \int_\Omega (\alpha_1^* \oplus \alpha_2^*)_+^2  + (\alpha_1^* \oplus \alpha_2^*)_-^2 \,\dd\lambda\\
  &\quad \leq 2\int_\Omega (\alpha_1^* \oplus \alpha_2^* - c)_+^2 + c_+^2 
  + (\alpha_1^*)_-^2 + (\alpha_2^*)_-^2 \,\dd\lambda < \infty,
 \end{aligned}
 \end{equation*}
 where we used \eqref{eq:alphaastfeas} and the boundedness of the negative parts proven above.
 Note that the constant shift, potentially needed to ensure $\int_{\Omega_2}\alpha_2^*\,\dd\lambda_2 = 0$ 
 has no effect on the equation in \eqref{eq:alphaastfeas} due to the additive structure of $\oplus$.
 
    We have thus shown that $(\alpha_1^*, \alpha_2^*)$ is feasible for \eqref{eq:qrot-dual-objective}. 
    Since $(\alpha_1^*, \alpha_2^*)$ solves \eqref{eq:dualmassprob}, 
    whose objective is the same as in \eqref{eq:qrot-dual-objective}, while its feasible set is larger,
    this implies that we have found a solution to \eqref{eq:qrot-dual-objective}. 
 \qed 
\end{proof}

We now show that, if $\pi^*$ is of the form $\pi^* = \gamma^{-1}(\alpha_1^* \oplus \alpha_2^* - c)_+$ 
with two functions $\alpha_i^* \in L^2(\Omega_i)$, $i=1,2$, 
and has the marginals $\mu_1$ and $\mu_2$, respectively, then it solves the necessary and sufficient optimality conditions 
of the primal problem \eqref{qrot-cont} in form of the following variational inequality:
\begin{equation}\tag{VI}\label{eq:vi}
    \pi^* \in \calF, \quad \scp{\gamma \pi^* + c}{\pi - \pi^*}_{L^2} \geq 0 \quad \forall\, \pi \in \calF.
\end{equation}
Herein, $\calF$ is the (convex) feasible set of \eqref{qrot-cont}, i.e.\
\begin{equation*}
\begin{aligned}
    \calF := \Big\{ \pi \in L^2(\Omega) : \pi \geq 0 \text{ $\lambda$-a.e.\ in } \Omega, \, 
    & \int_{\Omega_2} \pi \dd \lambda_2 = \mu_1 \text{ $\lambda_1$-a.e.\ in } \Omega_1, \\ 
    & \int_{\Omega_1} \pi \dd \lambda_1 = \mu_2 \text{ $\lambda_2$-a.e.\ in } \Omega_2 \Big\}.
\end{aligned}
\end{equation*}
For this purpose, let $\pi \in \calF$ be fixed but arbitrary. Multiplying the equality constraints in $\calF$ with $\alpha_1^*$ and $\alpha_2^*$, respectively, 
integrating the arising equations and add them yields
\begin{align}
    \int_{\Omega_1} \mu_1 \alpha_1^* \dd \lambda_1 + \int_{\Omega_2} \mu_2 \alpha_2^* \dd \lambda_2
    &= \int_{\Omega} \pi (\alpha_1^* \oplus \alpha_2^*) \dd \lambda \nonumber\\
    &= \int_\Omega \pi \big[(\alpha_1^* \oplus \alpha_2^* - c)_+ + c\big]\dd \lambda 
    - \int_\Omega \pi (\alpha_1^* \oplus \alpha_2^* - c)_-\dd \lambda \nonumber\\
    &\leq \int_\Omega \pi (\gamma \pi^* + c)\dd \lambda,\label{eq:vi1}
\end{align}
where we used $\pi \geq 0$ for the last inequality. Using the feasibility of $\pi^*$, we find similarly
\begin{align}    
    \int_{\Omega_1} \mu_1 \alpha_1^* \dd \lambda_1 + \int_{\Omega_2} \mu_2 \alpha_2^* \dd \lambda_2
    &= \int_{\Omega} \pi^* \big[(\alpha_1^* \oplus \alpha_2^* - c) +c\big]\dd \lambda \nonumber\\
    &= \int_\Omega \gamma^{-1}(\alpha_1^* \oplus \alpha_2^* - c)_+ \big[(\alpha_1^* \oplus \alpha_2^* - c) +  c\big]\dd \lambda \nonumber\\
    &= \int_\Omega \pi^* (\gamma \pi^* + c)\dd \lambda.\label{eq:vi2}
\end{align}
Combining \eqref{eq:vi1} and \eqref{eq:vi2} now yields \eqref{eq:vi}. 
As \eqref{qrot-cont} is a strictly convex minimization problem, this shows that, 
if $\pi^*$ has the form $\pi^* = \gamma^{-1}(\alpha_1^* \oplus \alpha_2^* - c)_+$ 
with functions $\alpha_i^* \in L^2(\Omega_i)$ and satisfies $\pi^* \in \calF$, then it is a solution of \eqref{qrot-cont}. On the other hand, 
we know from Theorem~\ref{thm:minimizers_dualproblem} that, under Assumption~\ref{assu:dual} 
(more or less needed for the existence of solutions of \eqref{qrot-cont} anyway), 
there always exist $\alpha_i^* \in L^2(\Omega_i)$ so that $\pi^* = \gamma^{-1}(\alpha_1^* \oplus \alpha_2^* - c)_+$ satisfies the 
equality constraints in $\calF$. Therefore, in summary we have deduced the following:

\begin{theorem}[Necessary and Sufficient Optimality Conditions for \eqref{qrot-cont}]\label{thm:optimality-ssn}
    Under Assumption~\ref{assu:dual}, $\pi^*\in L^2(\Omega)$ is a solution of \eqref{qrot-cont} 
    if and only if there exist functions $\alpha_i^* \in L^2(\Omega_i)$, $i=1,2$, such that the following 
    optimality system is fulfilled:
   \begin{subequations}
      \label{eq:kkt-qrot-ssn}
    \begin{alignat}{3}
      \pi^* - \tfrac1\gamma\left(\alpha_{1}^*\oplus\alpha_{2}^*-c\right)_{+}
      &= 0 &\quad  & \text{$\lambda$-a.e.~in }\Omega,\label{eq:kkt-qrot-ssn-1}\\
      \int_{\Omega_{2}}\left(\alpha_{1}^*\oplus\alpha_{2}^*-c\right)_{+} \,\dd \lambda_2  
      & =\gamma\mu_{1}  & & \text{$\lambda_1$-a.e.~in }\Omega_1,\label{eq:kkt-qrot-ssn-2}\\
      \int_{\Omega_{1}}\left(\alpha_{1}^*\oplus\alpha_{2}^*-c\right)_{+} \,\dd \lambda_{1} 
      & =\gamma\mu_{2} & & \text{$\lambda_2$-a.e.~in }\Omega_2.\label{eq:kkt-qrot-ssn-3}
    \end{alignat}
    \end{subequations}
\end{theorem}

The significance of Theorem~\ref{thm:optimality-ssn} lies in the fact that we can characterize optimality of $\pi$ 
by just two equalities in $L^{2}(\Omega_{1})$ and $L^{2}(\Omega_{2})$, respectively, namely \eqref{eq:kkt-qrot-ssn-2} 
and \eqref{eq:kkt-qrot-ssn-3}.
Thus, we effectively reduce the size of the problem from searching one function on $\Omega = \Omega_{1}\times \Omega_{2}$ 
to searching two functions, one on $\Omega_{1}$ and one on $\Omega_{2}$ (similarly as for entropic regularization, cf.~\cite{carlier2017convergence}).
This will be exploited numerically in Section~\ref{sec:algorithms}.

\subsection{Regularization of the dual problem}
\label{sec:reg-dual}

As seen before, the dual problem in \eqref{eq:qrot-dual-objective} is not uniquely solvable. 
One source of non-uniqueness is of course the kernel of the map $(\alpha_{1},\alpha_{2})\mapsto \alpha_{1}\oplus\alpha_{2}$. 
This kernel is one-dimensional and is spanned by the function $(1,-1)$, 
which could be easily taken into account in an algorithmic framework.
However, there is another source of non-uniqueness due to the max-operator that cuts of the negative part.
Here is a simple example where dual solutions are not unique: For $\Omega_1 = \Omega_2 = [0,1]$, $\mu_1 = \mu_2 \equiv 1$, and 
\begin{equation*}
  c(x,y) := \begin{cases}
      C, & \text{if } \frac{1}{2} \leq x \leq 1, \;  \frac{1}{2} \leq y \leq 1,\\
      0, & \text{else},
  \end{cases}
  \qquad \text{with} \quad C > 4,
\end{equation*}
one can show by a straight forward calculation that, for every $\delta \in [0,\frac{C-4}{2}]$, the tuple
\begin{equation*}
    \alpha_1^*(x) = \begin{cases}
        1 + \delta, & \text{if } x \in [0,\frac{1}{2}),\\
        -1 - \delta, & \text{if } x \in [\frac{1}{2}, 1],
    \end{cases}
    \qquad 
    \alpha_2^*(y) = \begin{cases}
        3 + \delta, & \text{if } y \in [0,\frac{1}{2}),\\
        1 - \delta, & \text{if } y \in [\frac{1}{2}, 1],
    \end{cases}
\end{equation*}
solves the optimality system \eqref{eq:kkt-qrot-ssn-2}--\eqref{eq:kkt-qrot-ssn-3}. This shows that the potential structure of 
non-uniqueness might become fairly intricate. 
A situation like this can certainly happen in the discretized problem we will derive in Section~\ref{sec:discrete-dual} and can lead to problems when we derive algorithms for the discrete problem since non-unique solutions imply a degenerate Hessian at the optimum.

Therefore, we investigate the following regularization of the dual problem:
\begin{equation}\tag{D$_\varepsilon$}
  \label{eq:dualeps}
  \left.
  \begin{aligned}
    \min \quad &  \Phi_\varepsilon(\alpha_{1},\alpha_{2}) \coloneqq \Phi(\alpha_{1},\alpha_{2}) 
    + \tfrac{\varepsilon}{2} \big(\|\alpha_1\|_{L^2(\Omega_1)}^2 + \|\alpha_2\|_{L^2(\Omega_2)}^2 \big)\\
    \text{s.t.} \quad & \alpha_i \in L^2(\Omega_i), \, i =1,2,
  \end{aligned}
  \quad \right\}
\end{equation}
with a regularization parameter $\varepsilon > 0$. It is clear that the additional quadratic terms in 
the regularized objective $\Phi_\varepsilon$ yield that the latter is strictly convex and coercive in 
$L^2(\Omega_1)\times L^2(\Omega_2)$. Therefore, for every $\varepsilon>0$, \eqref{eq:dualeps} admits a 
unique solution.

\begin{proposition}\label{prop:dualconv}
 Let $\{\varepsilon_n\} \subset \RR^+$ be a sequence converging to zero and denote the solutions of \eqref{eq:dualeps}
 with $\varepsilon = \varepsilon_n$
 by $(\alpha_1^n, \alpha_2^n) \in L^2(\Omega_1)\times L^2(\Omega_2)$.
 Then the sequence $\{(\alpha_1^n, \alpha_2^n)\}$ admits a weak accumulation point, every weak accumulation point is also strong one and
 a solution of the original dual problem \eqref{eq:qrot-dual-objective}.
\end{proposition}

\begin{proof} Let $(\alpha_1^*, \alpha_2^*)\in L^2(\Omega_1)\times L^2(\Omega_2)$ 
denote an arbitrary globally optimal solution of \eqref{eq:qrot-dual-objective} 
(whose existence is guaranteed by Theorem~\ref{thm:minimizers_dualproblem}). Then the optimality of 
$(\alpha_1^*, \alpha_2^*)$ for \eqref{eq:qrot-dual-objective} and of $(\alpha_1^n, \alpha_2^n)$ for \eqref{eq:dualeps}
(with $\varepsilon = \varepsilon_n$) gives
 \begin{equation*}
  \Phi(\alpha_1^*, \alpha_2^*) 
  + \tfrac{\varepsilon_n}{2} \big(\|\alpha_1^n\|_{L^2(\Omega_1)}^2 + \|\alpha_2^n\|_{L^2(\Omega_2)}^2 \big) \leq  \Phi_{\varepsilon_n}(\alpha_1^n, \alpha_2^n)
  \leq \Phi_{\varepsilon_n}(\alpha_1^*, \alpha_2^*)
 \end{equation*}
which implies
\begin{equation}\label{eq:alpharegbound}
 \|\alpha_1^n\|_{L^2(\Omega_1)}^2 + \|\alpha_2^n\|_{L^2(\Omega_2)}^2  
 \leq \|\alpha_1^*\|_{L^2(\Omega_1)}^2 + \|\alpha_2^*\|_{L^2(\Omega_2)}^2.
\end{equation}
Thus, the boundedness of $\{(\alpha_1^n, \alpha_2^n)\}$ in $L^2(\Omega_1)\times L^2(\Omega_2)$. 
This in turn gives the existence of a weak accumulation point as claimed.

Now assume that $(\tilde\alpha_1, \tilde\alpha_2)$ is such a weak accumulation point, i.e.
\begin{equation}\label{eq:alphaweak}
 (\alpha_1^n, \alpha_2^n) \weak (\tilde\alpha_1, \tilde\alpha_2) \quad \text{in } L^2(\Omega_1)\times L^2(\Omega_2)
\end{equation}
(for a subsequence).
Using again the optimality of $(\alpha_1^*, \alpha_2^*)$ and $(\alpha_1^n, \alpha_2^n)$, respectively, we obtain
 \begin{equation}\label{eq:dualobjconv}
 \begin{aligned}
  \Phi(\alpha_1^*, \alpha_2^*) 
  \leq \Phi(\alpha_1^n, \alpha_2^n) 
  \leq \Phi_{\varepsilon_n}(\alpha_1^n, \alpha_2^n) 
  \leq \Phi_{\varepsilon_n}(\alpha_1^*, \alpha_2^*) \to  \Phi(\alpha_1^*, \alpha_2^*).
 \end{aligned}
 \end{equation}
On the other hand, by convexity and weak lower semicontinuity of $\Phi$ we get from \eqref{eq:alphaweak} and \eqref{eq:dualobjconv} that 
 \begin{equation*}
  \Phi(\tilde\alpha_1, \tilde\alpha_2) \leq \liminf_{n\to\infty}  \Phi(\alpha_1^n, \alpha_2^n) 
  = \lim_{n\to\infty} \Phi(\alpha_1^n, \alpha_2^n)  = \Phi(\alpha_1^*, \alpha_2^*),
 \end{equation*}
 which gives in turn the optimality of the weak limit.
 Estimate \eqref{eq:alpharegbound} 
 for the choice $(\alpha_1^*, \alpha_2^*) = (\tilde{\alpha}_1, \tilde{\alpha}_2)$ shows that
 \[ \|\alpha_1^n\|_{L^2(\Omega_1)}^2 + \|\alpha_2^n\|_{L^2(\Omega_2)}^2
    \le \|\tilde{\alpha}_1\|_{L^2(\Omega_1)}^2 + \|\tilde{\alpha}_2\|_{L^2(\Omega_2)}^2 \]
 and thus, we have 
 \[ \liminf_{n\to\infty} \|\alpha_1^n\|_{L^2(\Omega_1)}^2 + \|\alpha_2^n\|_{L^2(\Omega_2)}^2 \le \|\tilde{\alpha}_1\|_{L^2(\Omega_1)}^2 + \|\tilde{\alpha}_2\|_{L^2(\Omega_2)}^2, \]
 but  $(\alpha_1^n, \alpha_2^n) \nrightarrow (\tilde{\alpha}_1, \tilde{\alpha}_2)$ would imply
 \[ \|\tilde{\alpha}_1\|_{L^2(\Omega_1)}^2 + \|\tilde{\alpha}_2\|_{L^2(\Omega_2)}^2 
 < \liminf_{n\to\infty} \|\alpha_1^n\|_{L^2(\Omega_1)}^2 + \|\alpha_2^n\|_{L^2(\Omega_2)}^2 \] 
 and consequently, we have 
 $(\alpha_1^n, \alpha_2^n) \to (\tilde{\alpha}_1, \tilde{\alpha}_2)$ in $L^2(\Omega_1) \times L^2(\Omega_2)$.
 \qed
\end{proof}

\begin{theorem}
 Let $\{\varepsilon_n\} \subset \RR^+$ be a sequence converging to zero and denote the solutions of \eqref{eq:dualeps}
 with $\varepsilon = \varepsilon_n$ again by $(\alpha_1^n, \alpha_2^n) \in L^2(\Omega_1)\times L^2(\Omega_2)$.
 Moreover, define
 \begin{equation}\label{eq:pindef}
  \pi_n \coloneqq \tfrac{1}{\gamma} (\alpha_1^n \oplus \alpha_2^n - c)_+.
 \end{equation}
 Then $\pi_n$ converges strongly in $L^2(\Omega)$ to the unique solution of \eqref{qrot-cont}.
\end{theorem}

\begin{proof}
 From \eqref{eq:alpharegbound}, we know that $\{(\alpha_1^n, \alpha_2^n)\}$ is bounded and hence,
 $\{\pi_n\}$ is bounded in $L^2(\Omega)$. Thus, 
 \begin{equation}\label{eq:pinweak}
  \pi_n \weak \tilde\pi \quad \text{in } L^2(\Omega)
 \end{equation}
 for some subsequence.
 Now we show that $\tilde\pi$ is the optimal for \eqref{qrot-cont}. Weak closedness of 
 $\{\pi \in L^2(\Omega) : \pi(x_1, x_2) \geq 0 \text{ a.e.~in }\Omega\}$
 implies $\tilde\pi\geq 0$. 
 Integrating the  first-order optimality conditions for \eqref{eq:dualeps}
 \begin{alignat}{3}
  \int_{\Omega_2} (\alpha_1^n \oplus \alpha_2^n - c)_+ \,\dd\lambda_2  + \varepsilon \,\alpha_1^n &= \gamma\,\mu_1 
  & \quad  & \text{$\lambda_1$-a.e.~in } \Omega_1 \label{eq:noceps1}\\
  \int_{\Omega_1} (\alpha_1^n \oplus \alpha_2^n - c)_+ \,\dd\lambda_1  + \varepsilon \,\alpha_2^n &= \gamma\,\mu_2
  & & \text{$\lambda_2$-a.e.~in } \Omega_2. \label{eq:noceps2}
 \end{alignat}
 against some $\varphi_1 \in C^\infty_c(\Omega_1)$, inserting the definition of $\pi_n$,  and integrating over $\Omega_1$ yields
 \begin{equation*}
  \int_{\Omega_1} \int_{\Omega_2} \pi_n\,\dd\lambda_2 \,\varphi_1\,\dd\lambda_1 
  = \int_{\Omega_1} \mu_1\,\varphi_1\,\dd\lambda_1 - \frac{\varepsilon_n}{\gamma} \int_{\Omega_1} \alpha_1^n\,\varphi_1\,\dd\lambda_1
 \end{equation*}
 Passing to the limit we obtain
 \begin{equation*}
  \int_{\Omega_1} \int_{\Omega_2} \tilde \pi \,\dd\lambda_2 \,\varphi_1 \,\dd\lambda_1 = \int_{\Omega_1} \mu_1\,\varphi_1 \,\dd\lambda_1,
 \end{equation*}
 and thus, $\tilde\pi$ satisfies the first equality constraint in \eqref{qrot-cont}. 
 The second equality constraint can be verified analogously.
 
 To show optimality of $\tilde\pi$, we test the optimality conditions in \eqref{eq:noceps1} and \eqref{eq:noceps2} with $\alpha_1^n$ and $\alpha_2^n$, respectively, and get
 \begin{equation*}
 \begin{aligned}
   \Phi_{\varepsilon_n} (\alpha_1^n, \alpha_2^n)
   &= \tfrac{\gamma^2}{2} \|\pi_n\|_{L^2(\Omega)}^2 - \gamma \int_\Omega \pi_n (\alpha_1^n \oplus \alpha_2^n) \,\dd \lambda 
   - \tfrac{\varepsilon_n}{2} \|\alpha_1^n\|_{L^2(\Omega_1)}^2 - \tfrac{\varepsilon_n}{2} \|\alpha_2^n\|_{L^2(\Omega_2)}^2\\
   &= - \tfrac{\gamma^2}{2} \|\pi_n\|_{L^2(\Omega)}^2 - \gamma \int_\Omega c\,\pi_n\,\dd\lambda 
  - \tfrac{\varepsilon_n}{2} \|\alpha_1^n\|_{L^2(\Omega_1)}^2 - \tfrac{\varepsilon_n}{2} \|\alpha_2^n\|_{L^2(\Omega_2)}^2 \\ 
   & = - \gamma E_\gamma(\pi_n) - \tfrac{\varepsilon_n}{2} \|\alpha_1^n\|_{L^2(\Omega_1)}^2 - \tfrac{\varepsilon_n}{2} \|\alpha_2^n\|_{L^2(\Omega_2)}^2,
 \end{aligned}
 \end{equation*}
 where $E_\gamma$ is the primal objective from \eqref{eq:primobj}. Similarly, we get
 \begin{equation*}
   \Phi(\alpha_1^*, \alpha_2^*) = - \gamma\, E_\gamma (\pi^*),
 \end{equation*}
 where $\pi^*\in L^2(\Omega)$ is the unique solution of \eqref{qrot-cont} and $(\alpha_1^*, \alpha_2^*)\in L^2(\Omega_1) \times L^2(\Omega_2)$ 
 solves the dual problem \eqref{eq:qrot-dual-objective}. Now, putting everything so far together, we obtain
 \begin{equation*}
 \begin{aligned}
  \lim_{n\to\infty} E_\gamma(\pi_n) &= 
  \lim_{n\to\infty} \Big(- \tfrac{1}{\gamma} \,\Phi_{\varepsilon_n} (\alpha_1^n, \alpha_2^n) 
  - \tfrac{\varepsilon_n}{2\gamma} \|\alpha_1^n\|_{L^2(\Omega_1)}^2 - \tfrac{\varepsilon_n}{2\gamma} \|\alpha_2^n\|_{L^2(\Omega_2)}^2\Big)\\
  &= - \tfrac{1}{\gamma} \,\Phi (\alpha_1^*, \alpha_2^*) = E_\gamma(\pi^*).
 \end{aligned}  
 \end{equation*}
 On the other hand, $E_\gamma$ is  weakly lower semicontinuous, and therefore
 \begin{equation*}
  E_\gamma(\tilde\pi) \leq \liminf_{n\to\infty} E_\gamma(\pi_n) = E_\gamma(\pi^*).
 \end{equation*}
 This gives the optimality of $\tilde\pi$ and by strict convexity also uniqueness, i.e. $\tilde\pi = \pi^*$. Thus, the weak limit is unique and a well known argument by contradiction therefore
 implies the weak convergence of the whole sequence $\{\pi_n\}$ to $\pi^*$. 
 Finally, strong convergence follows from a standard argument.
 \qed
\end{proof}

\subsection{The discrete dual problem}
\label{sec:discrete-dual}

We show a simple discretization of the quadratically
regularized optimal transport problem~\eqref{qrot-cont} by piecewise constant approximation in
Appendix~\ref{sec:discretization}. To keep the notation concise, we state the corresponding discrete optimal transport problem and  illustrate the duality already here. This will be the basis of our
algorithms we derive in Section~\ref{sec:algorithms}. A discrete version of the
continuous problem~\eqref{qrot-cont} is the finite-dimensional problem 
\begin{equation}
  \label{eq:qrot-discrete}
  \min_{\pi\in\RR^{M\times N}}\scp{\pi}{c} + \tfrac\gamma2\norm{\pi}_{F}^{2}\quad\text{s.t.}\quad\pi^{T}\one_{M} = \mu,\ \pi\one_{N} = \nu,\ \pi\geq 0
\end{equation}
where $\one_{N}\in\RR^{N}$ denotes the vector of all ones,
$\mu\in\RR^{N}$, $\nu\in\RR^{M}$ denote the discretized
marginals with
$\sum_{j=1}^{N}\mu_{j} = \sum_{i=1}^{M}\nu_{j}$, and
$c\in\RR^{M\times N}$ denotes the discretized cost. Note that we slightly changed the notation from $\mu_{1}$ and $\mu_{2}$ to $\mu$ and $\nu$, respectively. For the discrete
form of the optimality system~\eqref{eq:kkt-qrot-ssn}  we further replace the Lagrange multipliers $\alpha_{1}$ and $\alpha_{2}$ by $\alpha$
and $\beta$, respectively, and get
\begin{subequations}
  \label{eq:kkt-qrot-ssn-discrete}
    \begin{align}
      \pi & = \tfrac1\gamma\left(\alpha\oplus\beta-c\right)_{+}\label{eq:kkt-qrot-ssn-discrete-1}\\
      \sum_{i=1}^{M}\left(\alpha_{i}+\beta_{j}-c_{ij}\right)_{+}  & =\gamma\mu_{j},\ j=1,\dots,N\label{eq:kkt-qrot-ssn-discrete-2}\\
      \sum_{j=1}^{N}\left(\alpha_{i}+\beta_{j}-c_{ij}\right)_{+}  & =\gamma\nu_{i},\ i=1,\dots,M\label{eq:kkt-qrot-ssn-discrete-3}
    \end{align}
\end{subequations}
where $\alpha\in\RR^{M}$, $\beta\in\RR^{N}$ and
$(\alpha\oplus\beta)_{i,j} = \alpha_{i} + \beta_{j}$ is the ``outer
sum''. The discrete counterpart of $\Phi$ from~\eqref{eq:qrot-dual-objective} is
\[
\Phi(\alpha,\beta) = \tfrac12\norm{(\alpha\oplus\beta - c)_{+}}_{F}^{2} - \gamma\scp{\nu}{\alpha} - \gamma\scp{\mu}{\beta}
\]
where $\norm{\cdot}_{F}$ denotes the Frobenius norm.

We write the optimality condition~\eqref{eq:kkt-qrot-ssn-discrete-2}-\eqref{eq:kkt-qrot-ssn-discrete-3} as a non-smooth equation $F(\alpha,\beta) = 0$ in $\RR^{M+N}$ with
\begin{equation}\label{eq:opt-cond-F}
  F(\alpha,\beta) =
  \begin{pmatrix} F_{1}(\alpha,\beta)\\ F_{2}(\alpha,\beta)
  \end{pmatrix} =
  \begin{pmatrix}
    (\sum_{j=1}^{N}\left(\alpha_{i}+\beta_{j}-c_{ij}\right)_{+} -
    \gamma\nu_{i})_{i=1,\dots,M}\\
    (\sum_{i=1}^{M}\left(\alpha_{i}+\beta_{j}-c_{ij}\right)_{+}
    -\gamma\mu_{j})_{j=1,\dots,N}
  \end{pmatrix}
\end{equation}
(note that $F_{1} = \partial_{\alpha}\Phi$ and
$F_{2} = \partial_{\beta}\Phi$).
Since $F$ is the composition of Lipschitz continuous and semismooth functions, we have the following result 
(for the chain rule for semismooth functions, see e.g.\ \cite[Thm. 2.10]{OPTpdecon}): 
\begin{lemma}\label{lem:lipschitz-gradient-Phi}
  The function $F$ (and thus, the gradient of $\Phi$) is (globally) Lipschitz continuous and semismooth.
\end{lemma}

\section{Algorithms}
\label{sec:algorithms}

The optimality system~\eqref{eq:kkt-qrot-ssn-discrete-2},~\eqref{eq:kkt-qrot-ssn-discrete-3} for the smooth and convex
problem~\eqref{eq:qrot-dual-objective} can be solved by different
methods. In~\cite{blondel2018smoothOT} the authors propose to use a
generic L-BFGS solver and also derive an alternating minimization
scheme, which is similar to the non-linear Gauss-Seidel method in the next
section, but differs slightly in the numerical realization
and~\cite{roberts2017gini} also uses an off-the-shelf solver. Here we
propose methods that exploit the special structure of the
optimality system: A non-linear Gauss-Seidel method and a semismooth
Newton method.

\subsection{Non-linear Gauss-Seidel}
\label{sec:nl-gs}

The method in this section is similar to the one described in the
Appendix of~\cite{blondel2018smoothOT}, but we describe it here for
the sake of completeness. A close look at the optimality system
\begin{subequations}
\begin{alignat}{3}
  \sum_{j=1}^{N}(\alpha_{i}+\beta_{j}-c_{ij})_{+} &=
  \gamma\nu_{i}, & \quad & i=1,\dots,M. \label{eq:opt-cond-minus}\\
  \sum_{i=1}^{M}(\alpha_{i}+\beta_{j}-c_{ij})_{+} &=
  \gamma\mu_{j}, & \quad & j=1,\dots,N \label{eq:opt-cond-plus}
\end{alignat}
\end{subequations}
shows that we can solve all $M$ equations in~\eqref{eq:opt-cond-minus}
for the $\alpha_{i}$ in parallel (for fixed $\beta$) since the $i$th
equation depends on $\alpha_{i}$ only. Similarly, all $N$ equations
in~\eqref{eq:opt-cond-plus} can be solved for the $\beta_{j}$ if
$\alpha$ is fixed. Hence, we can perform a non-linear Gauss-Seidel
method for these non-smooth equations (also known as alternating minimization, nonlinear SOR or coordinate descent method for $\Phi$~\cite{chen2002convergence,wright2015coordinatedescent}), i.e. alternatingly solving the
equations~\eqref{eq:opt-cond-minus} for $\alpha$ (for fixed $\beta$)
and then the equations~\eqref{eq:opt-cond-plus} for $\beta$ (for fixed
$\alpha$). The whole method is stated in Algorithm~\ref{alg:qrot_gs}.
Since $\Phi$ is convex with Lipschitz continuous gradient (cf.~Lemma~\ref{lem:lipschitz-gradient-Phi}) the convergence of the algorithm follows from results in~\cite{bertsekas2016nlp}.

\begin{algorithm}
\caption{Non-linear Gauss-Seidel for quadratically regularized optimal transport}\label{alg:qrot_gs}
\begin{algorithmic}
\State Initialize: $\beta^{0}\in\RR^{N}$, set $k=0$
\Repeat
\State Set $\alpha^{k+1}$ to be the solution of~\eqref{eq:opt-cond-minus} with $\beta = \beta^{k}$.
\State Set $\beta^{k+1}$ to be the solution of~\eqref{eq:opt-cond-plus} with $\alpha = \alpha^{k+1}$.
\State $k \gets k+1$
\Until{some stopping criterion}
\end{algorithmic}
\end{algorithm}

Each equation for an $\alpha_{i}$ or $\beta_{j}$ is just a single scalar equation for a scalar quantity and the structure of the equation is of the following form: For a given vector $y\in\RR^{n}$ and right hand side $b\in\RR$, solve
\begin{equation}
  \label{eq:max-equation}
  f(x) \coloneqq \sum_{j=1}^{n}(x-y_{j})_{+} = b.
\end{equation}
Of course, one can solve this problem by bisection, but here are two
other, more efficient methods to solve equations of the type~\eqref{eq:max-equation}:
\begin{description}
\item[\textbf{Direct search.}]
  If we
  denote by $y_{[j]}$ the $j$-th smallest entry of $y$ (i.e. we sort
  $y$ in an ascending way), we get that
  \[
  \begin{split}
    f(x) & = \sum_{j=1}^{n}(x-y_{[j]})_{+}\\
    & =
    \begin{cases}
      0, & x\leq y_{[1]}\\
      kx - \sum_{j=1}^{k}y_{[j]}, & y_{[k]}\leq x\leq y_{[k+1]},\ k=1,\dots,n-1\\
      nx - \sum_{j=1}^{n}y_{[j]}, & x\geq y_{[n]}.
    \end{cases}
  \end{split}
  \]
  To obtain the solution of~\eqref{eq:max-equation} we evaluate $f$ at
  the break points $y_{[j]}$ until we find the interval
  $[y_{[k]},y_{[k+1]}[$ in which the solution lies (by finding $k$
  such that $f(y_{[k]})\leq b< f(y_{[k+1]})$), and then setting
  \[
  x = \frac{b +\sum_{j=1}^{k}y_{[j]}}k.
  \]
  The complexity of the method is dominated by the sorting of the vector $y$, its complexity is $\bigO(n\log(n))$.
\item[\textbf{Semismooth Newton.}]  Although $f$ is
  non-smooth, we may perform Newton's method here. The function $f$ is
  piecewise linear and on each interval $]y_{[j]},y_{[j+1]}[$ is has
  the slope $j$ (a simple situation with $n=3$ is shown in
  Figure~\ref{fig:gs-f-ssn}). At the break points we may define $f'(y_{[j]}) = j$ and then we iterate
  \[
  x^{k+1} = x^{k} - \tfrac{f(x^{k})}{f'(x^{k})}.
  \]
  If we start with $x^{0}\geq y^{[n]} = \max_{k}y_{k}$, the method
  will produce a monotonically decreasing sequence which converges
  in at most $n$ steps. Actually, we can initialize the method with
  any $x^{0}$ that is strictly larger than $y_{[1]} = \min_{k}y_{k}$.
  
  Note that we do not need to sort the values of $y_{k}$ to calculate
  the derivative since we have $f'(x) = \#\{i\ :\ x\geq y_{i}\}$. In
  practice, the method usually needs much less iterations than $n$.
\end{description}

\begin{figure}[htb]
  \centering
  \begin{tikzpicture}
    \draw[->] (-2,0) -- (4,0)node[below]{$x$};
    \draw[->] (0,-2) -- (0,2.5)node[right]{$f(x)$};
    \draw (1,0.05) -- (1,-0.05)node[above]{$y_{[1]}$};
    \draw (2,0.05) -- (2,-0.05)node[above]{$y_{[2]}$};
    \draw (2.5,0.05) -- (2.5,-0.05)node[below]{$y_{[3]}$};
    \draw[dashed,gray] (1,0) -- (1,-1.5);
    \draw[dashed,gray] (2,0) -- (2,-0.5);
    \draw[dashed,gray] (2.5,0) -- (2.5,0.5);
    
    \draw[thick] (-1.9,-1.5) -- (1,-1.5) -- (2,-0.5) -- (2.5,0.5) -- (3,2);
    
  \end{tikzpicture}
  \caption{Illustration of the non-smooth function $f$ from~\eqref{eq:max-equation}.}
  \label{fig:gs-f-ssn}
\end{figure}
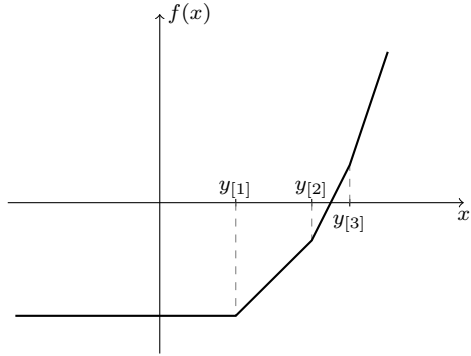

\subsection{Semismooth Newton}
\label{sec:ssn}

As seen in Lemma~\ref{lem:lipschitz-gradient-Phi}, the mapping $F$ is semismooth and hence, we may use a
semismooth Newton method~\cite{chen1997,chen2000semismooth}.

A simple calculation proves the following lemma.
\begin{lemma}\label{lem:newton-derivative}
  A Newton derivative of $F$ from~\eqref{eq:opt-cond-F} at $(\alpha,\beta)$ is given by
  \[
  G =
  \begin{pmatrix}
    \diag(\sigma\one_{N}) & \sigma\\
    \sigma^{T} & \diag(\sigma^{T}\one_{M})
  \end{pmatrix}\in\RR^{(M+N)\times (M+N)}
  \]
  where $\sigma\in\RR^{M\times N}$ is given by
  \[
  \sigma_{ij} =
  \begin{cases}
    1 & \alpha_{i} + \beta_{j}-c_{ij}\geq 0\\
    0 & \text{otherwise.}
  \end{cases}
  \]
\end{lemma}

A step of the semismooth Newton method for the solution of $F(\alpha,\beta) = 0$ would consist of
setting
\[
\begin{pmatrix}
  \alpha^{k+1}\\\beta^{k+1}
\end{pmatrix}
=
\begin{pmatrix}
  \alpha^{k}\\\beta^{k}
\end{pmatrix}
-
\begin{pmatrix}
  \delta\alpha^{k}\\\delta\beta^{k}
\end{pmatrix}
\quad\text{where}
\quad
F(\alpha^{k},\beta^{k}) = G
\begin{pmatrix}
  \delta\alpha^{k}\\\delta\beta^{k}
\end{pmatrix}.
\]
However, the next lemma shows, that $G$ has a non-trivial kernel.
\begin{lemma}\label{lem:newton-matrix}
 Let $G$ be the Newton derivative of $F$ at $(\alpha,\beta)$ defined in Lemma~\ref{lem:newton-derivative}.
 Then the following holds true:
 \begin{enumerate}
  \item $G \in \RR^{(M+N)\times (M+N)}$ is symmetric, 
  \item $G$ is positive semi-definite,
  \item $(a,b)\in\Kern(G)$ if and only if $\sigma_{ij}(a_{i}+b_{j}) = 0$ for all $1\leq i\leq M$, $1\leq j\leq N$.
\end{enumerate}  
\end{lemma}

\begin{proof}
 Symmetry of $G$ is clear by construction.
 To see that $G$ is positive semi-definite we calculate
 \begin{equation*}
 \begin{aligned}
  (a,b)^\top G (a,b)
  &= \sum_{j=1}^{N}\sum_{i=1}^M \sigma_{ij }a_i^2 + \sum_{j=1}^N\sum_{i=1}^{M} \sigma_{ij} b_j^2 
  + 2 \sum_{j=1}^{N}\sum_{i=1}^M\sigma_{ij}  a_i b_j\\
  & = \sum_{j=1}^{N}\sum_{i=1}^{M}\sigma_{ij}(a_{i}+b_{j})^{2}\geq 0.
 \end{aligned}
 \end{equation*}
 Due to the non-negativity of $\sigma$, this also shows the last point.\qed
\end{proof}
The third point of the lemma shows that the kernel of $G$ may have a
high dimension, depending on the matrix $\sigma$. Hence we resort to a quasi Newton method where we regularize the Newton step 
arising from the dual problem from Section~\ref{sec:dual-problem} by setting
\[
\begin{pmatrix}
  \alpha^{k+1}\\\beta^{k+1}
\end{pmatrix}
=
\begin{pmatrix}
  \alpha^{k}\\\beta^{k}
\end{pmatrix}
-
\begin{pmatrix}
  \delta\alpha^{k}\\\delta\beta^{k}
\end{pmatrix}
\quad\text{where}
\quad
F(\alpha^{k},\beta^{k}) = (G + \varepsilon I)
\begin{pmatrix}
  \delta\alpha^{k}\\\delta\beta^{k}
\end{pmatrix}
\]
with a small $\varepsilon>0$. By~\cite{chen1997}, the method still
converges, but only a local linear rate is guaranteed. We note that we have not applied the semismooth Newton method to the regularized dual problem 
from Section~\ref{sec:reg-dual}. This would also be possible, but lead not only to the regularized Newton matrix from above but we would also have 
to adapt the objective $F$ in the computation of the update.

Let us make a few remarks on the the regularized Newton step and its numerical treatment.
\begin{itemize}
\item The matrix $\sigma$ (and hence the Newton matrix $G$) is usually
  very sparse. The closer $\alpha$ and $\beta$ are to the optimal
  ones, the closer $(\alpha_{i}+\beta_{j}-c_{ij})_{+}$ is to the
  optimal regularized transport plan $\pi$ and for small $\gamma$ this
  usually very sparse.
\item Since $G$ is positive semi-definite, the regularized step could
  be done by the method of conjugate gradients. However, any linear
  solver that can exploit the sparsity of $G$ can be used.
\end{itemize}
As usual, the regularized semismooth Newton method may not converge
globally. A simple globalization technique is an Armijo linesearch in
the Newton direction. The full method is described in Algorithm~\ref{alg:qrot_ssn}.

\begin{algorithm}
\caption{Globalized and regularized semismooth Newton method quadratically regularized optimal transport}\label{alg:qrot_ssn}
\begin{algorithmic}
\State Initialize: $\alpha^{0}\in\RR^{M}$, $\beta^{0}\in\RR^{N}$, set $k=0$, choose regularization parameter $\varepsilon>0$, Armijo parameters $\theta,\kappa\in]0,1[$, and a tolerance $\tau>0$ 
\Repeat
\State Calculate
\[
P_{ij} = \alpha^{k}_{i}  + \beta^{k}_{j}-c_{ij},\quad\sigma_{ij} =
\begin{cases}
  1 & P_{ij}\geq 0\\
  0 & \text{otherwise}
\end{cases},
\quad\text{and}\quad
\pi_{ij} = \max(P_{ij},0)/\gamma.
\]
\State Calculate $\delta\alpha$ and $\delta\beta$ by solving
\[
\left(
\begin{pmatrix}
  \diag(\sigma\one_{N}) & \sigma\\
  \sigma^{T} & \diag(\sigma^{T}\one_{M})
\end{pmatrix}
+ \varepsilon I\right)
\begin{pmatrix}
  \delta\alpha\\\delta\beta
\end{pmatrix}
= -\gamma
\begin{pmatrix}
  \pi\one_{N} - \nu\\
  \pi^{T}\one_{M} - \mu
\end{pmatrix}
\]
\State Set $t=1$ and compute the directional derivative
\[
d = D_{(\delta\alpha,\delta\beta)}\Phi(\alpha^{k},\beta^{k}) = \gamma\sum_{ij}\pi_{ij}(\delta\alpha_{i}+\delta\beta_{j}) - \gamma(\scp{\delta\alpha}{\nu} + \scp{\delta\beta}{\mu}).
\]
\While {$\Phi(\alpha^{k}+t\delta\alpha,\beta^{k}+t\delta\beta) \geq \Phi(\alpha^{k},\beta^{k}) + t\theta d$}
\State $t\gets \kappa t$
\EndWhile
\State Set $\alpha^{k+1} = \alpha^{k} - t\delta\alpha$, $\beta^{k+1} = \beta^{k}-t\delta\beta$
\State $k \gets k+1$
\Until{$\norm{\pi\one_{N}-\nu}_{\infty},\norm{\pi^{T}\one_{M}-\mu}_{\infty}\leq\tau$}
\end{algorithmic}
\end{algorithm}

\section{Numerical examples}
\label{sec:examples}

\subsection{Illustration of $\gamma\to 0$}
\label{sec:num-gamma-conv}

In our first numerical example we illustrate the how the solutions $\pi^{*}$ of the regularized problem converge for vanishing regularization parameter $\gamma\to 0$. We generate some marginals,
fix a transport cost and compute solutions of the discretized
transport problems~\eqref{eq:qrot-discrete} for a sequence
$\gamma_{n}\to 0$ and illustrate the optimal transport plans (and the
related regularized transport costs). Our marginals are non-negative functions sampled at equidistant points $x_{i}$, $y_{i}$ in the interval $[0,1]$ and we used $M=N=400$ and the cost $c_{ij} = (x_{i}-y_{j})^{2}$ is the squared distance between the sampling points. The results are shown in Figure~\ref{fig:qrot_gamma_to_zero}.
One observes that the optimal transport plans converge to a measure that is singular and is supported on the graph of a monotonically increasing function, exactly as the fundamental theorem of optimal transport~\cite{ambrosio2013user} predicts.

\begin{figure}[htb]
  \centering
  \begin{tabular}{cccc}
  \includegraphics[width=0.23\textwidth]{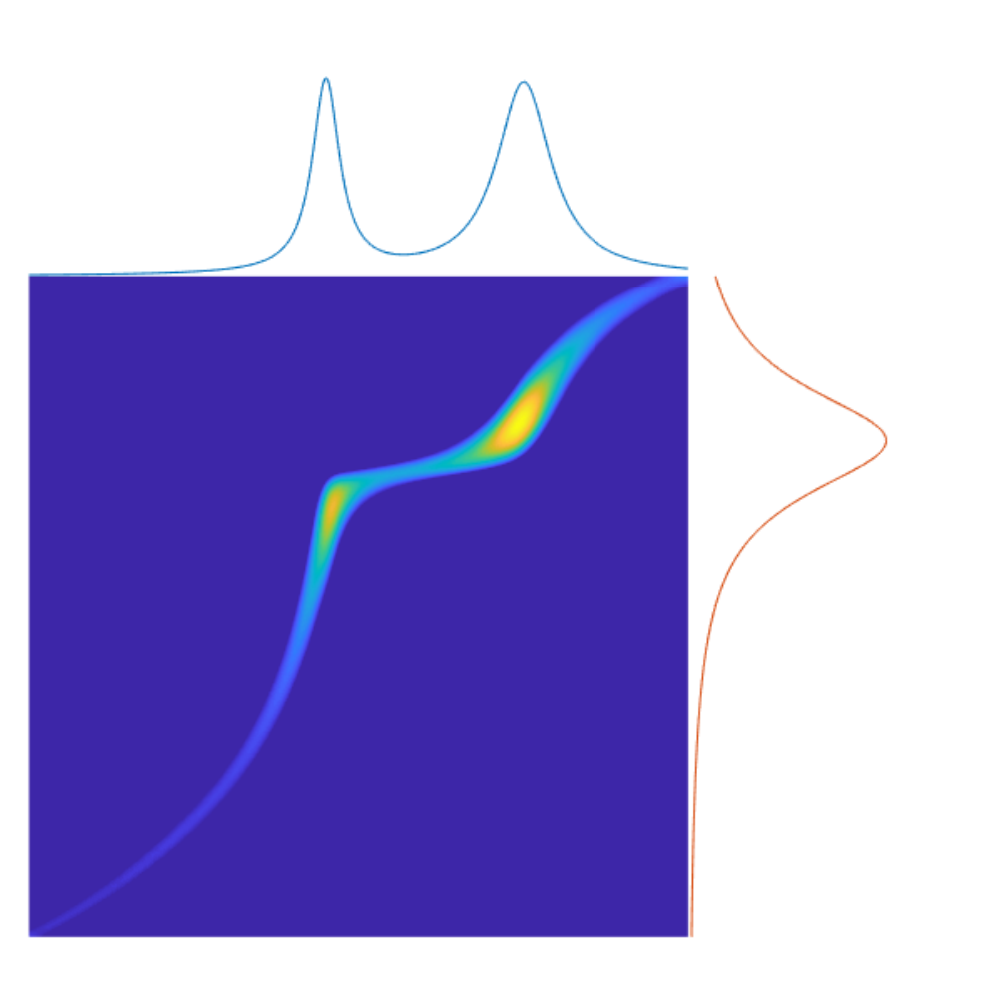}&
  \includegraphics[width=0.23\textwidth]{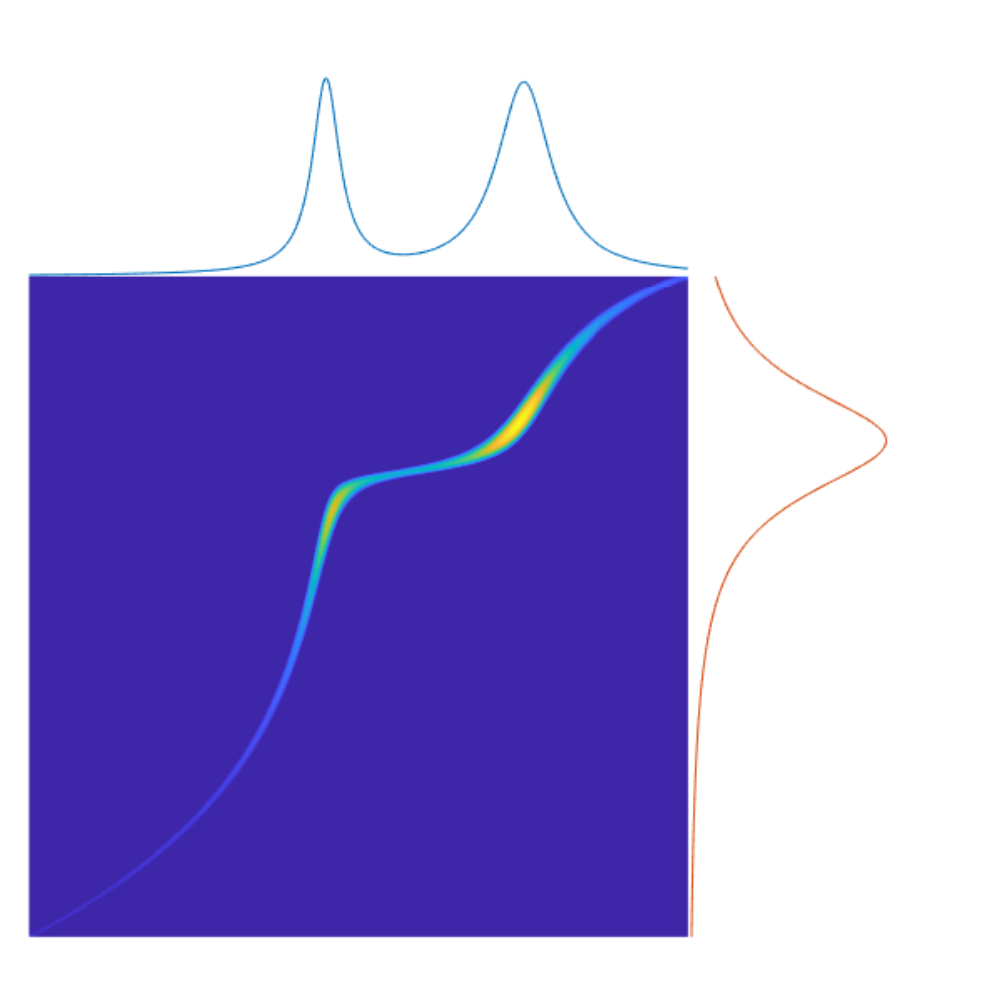}&
  \includegraphics[width=0.23\textwidth]{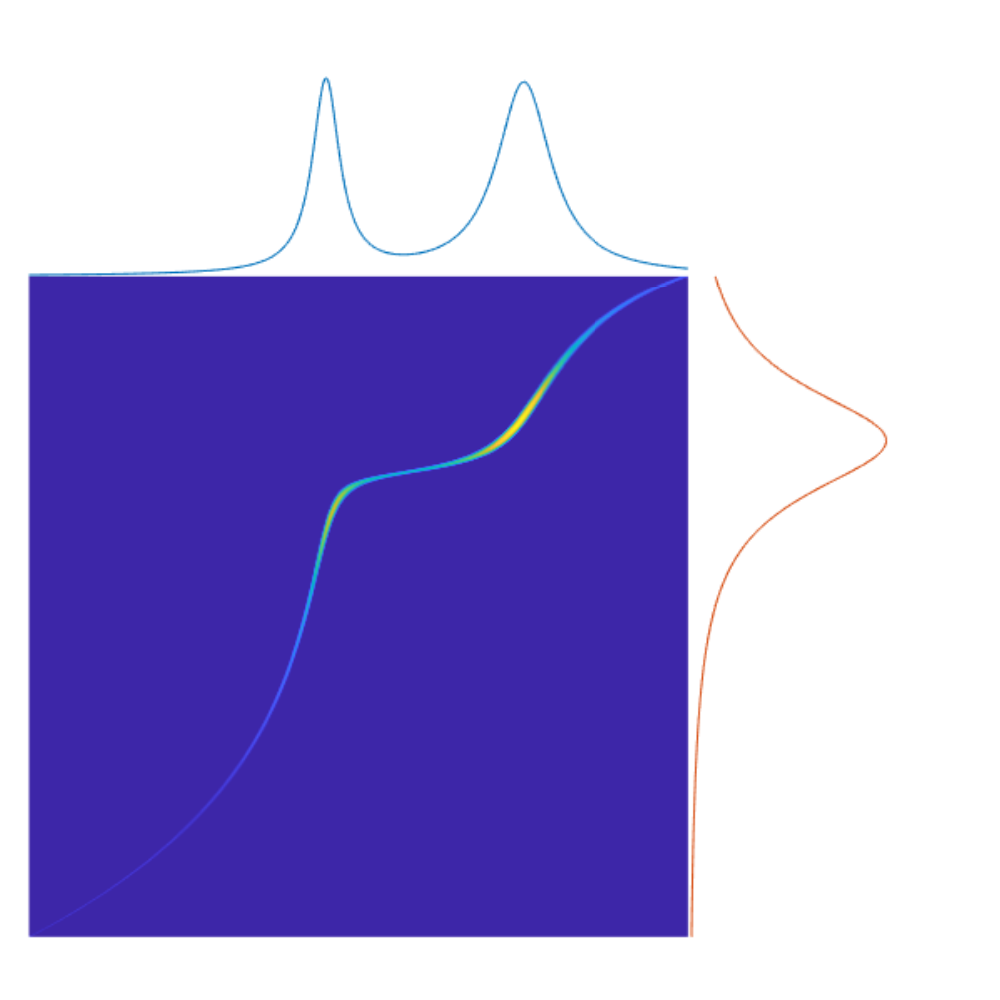}&
  \includegraphics[width=0.23\textwidth]{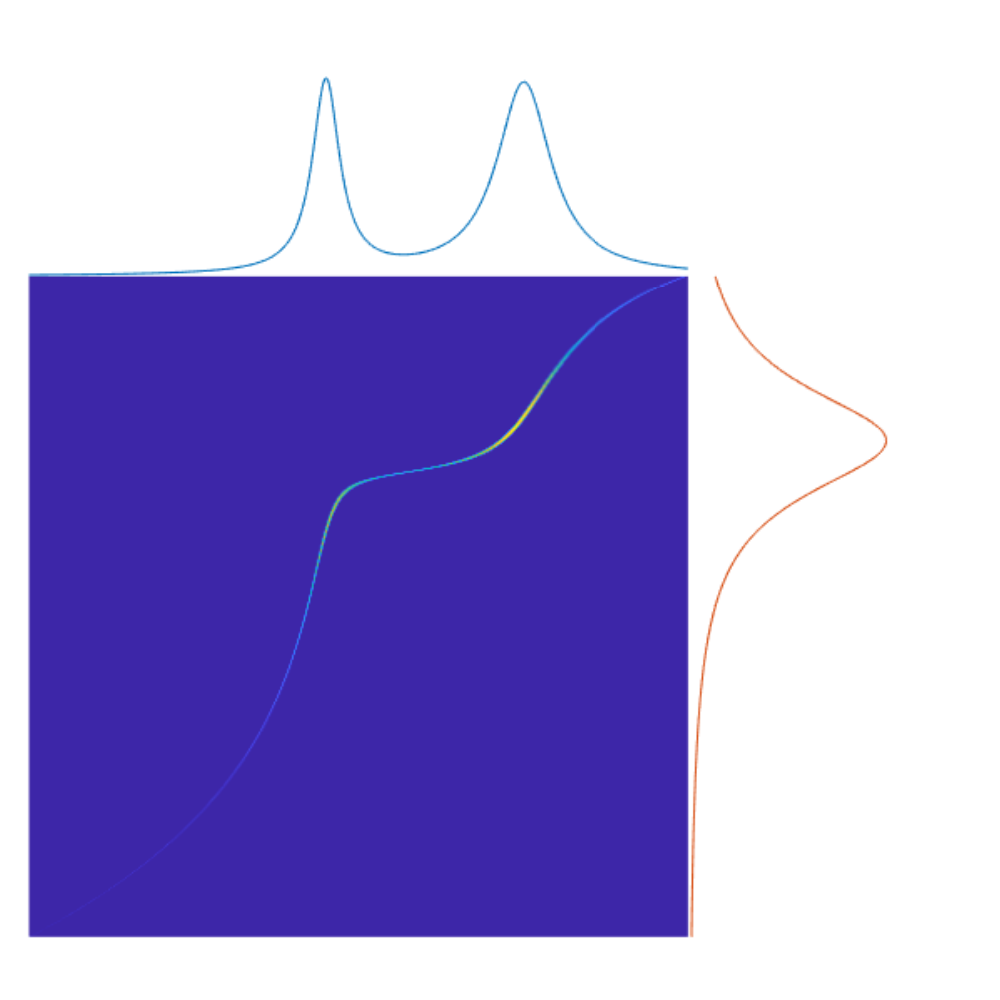}\\
                                                                        $\gamma = 10$ &
                                                                       $\gamma = 1$ &
                                                                        $\gamma = 0.1$&
                                                                                        $\gamma = 0.01$
  \end{tabular}
  \caption{Visualization of transport plans of the quadratically regularized optimal transport problem with $M=N=400$ and quadratic transport cost $c_{ij} = (x_{i}-y_{j})^{2}$.}
  \label{fig:qrot_gamma_to_zero}
\end{figure}

We repeat the same experiment where the cost is the (non-squared) distance $c_{ij} = \abs{x_{i}-y_{j}}$. Here we had to choose larger regularization parameters as it turned out that values similar to Figure~\ref{fig:qrot_gamma_to_zero} would lead to almost undistinguishable results. The results are shown in Figure~\ref{fig:qrot_gamma_to_zero_abs}. Note the different structure of the transport plan (which is again in agreement with the predicted results from the fundamental theorem of optimal transport). In Figure~\ref{fig:qrot_gamma_to_zero_sqrt} we show the results for the concave but increasing cost $c_{ij} = \sqrt{\abs{x_{i}-y_{j}}}$ and again observe the expected effect that a concave transport cost encouraged that as much mass as possible stays in place (as can be seen by the concentration of mass along the diagonal of the transport plan).

\begin{figure}[htb]
  \centering
  \begin{tabular}{cccc}
  \includegraphics[width=0.23\textwidth]{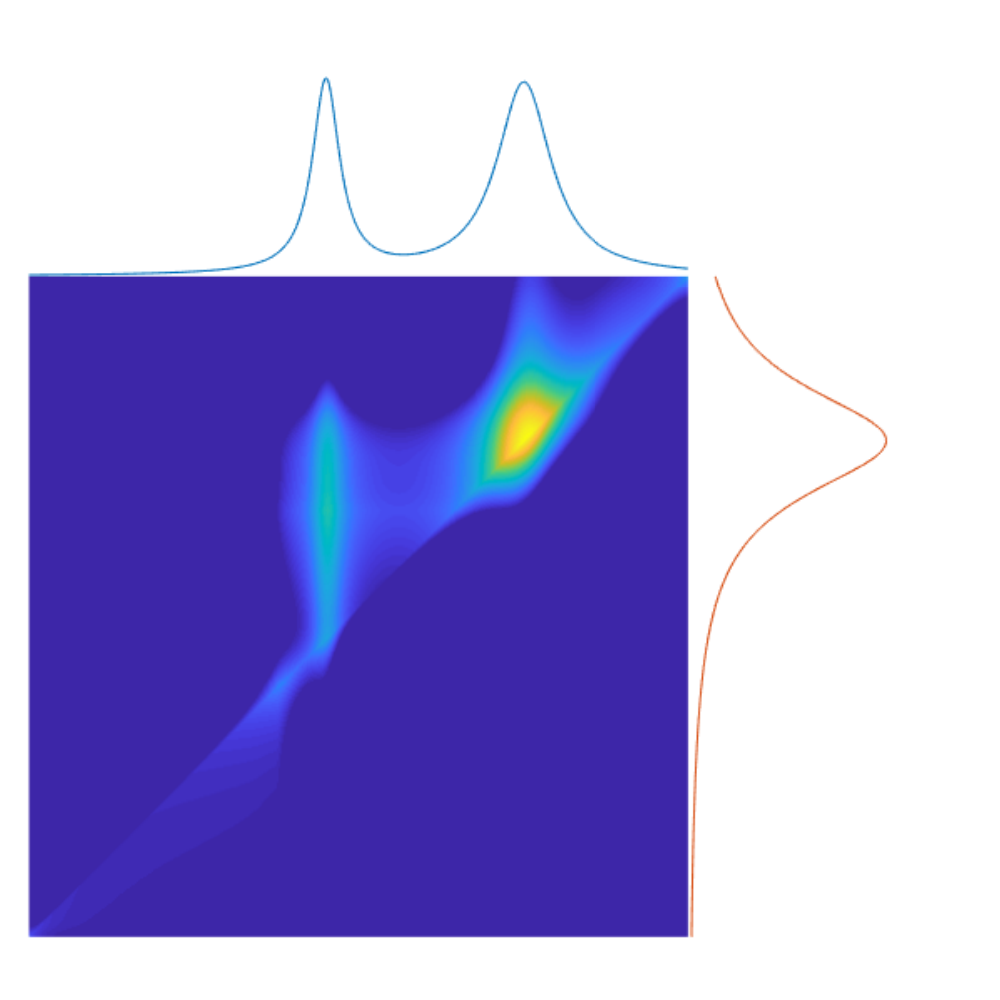}&
  \includegraphics[width=0.23\textwidth]{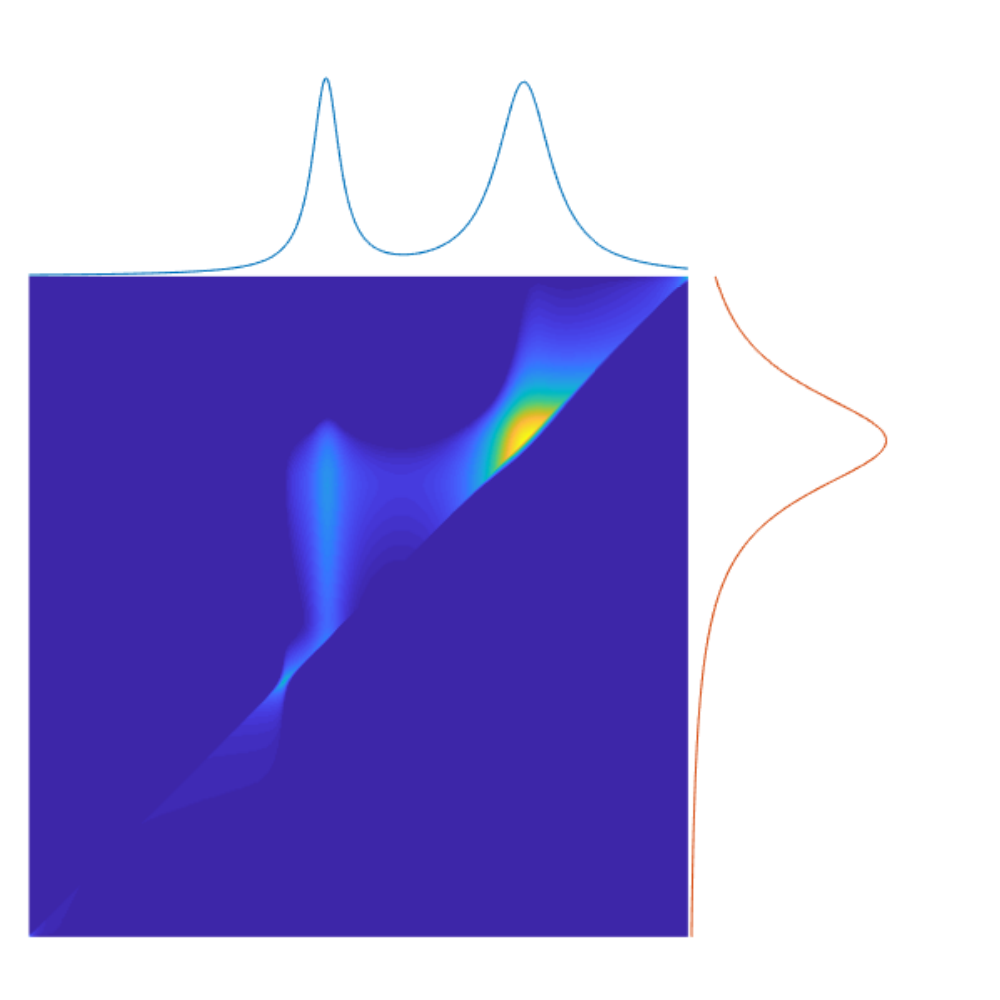}&
  \includegraphics[width=0.23\textwidth]{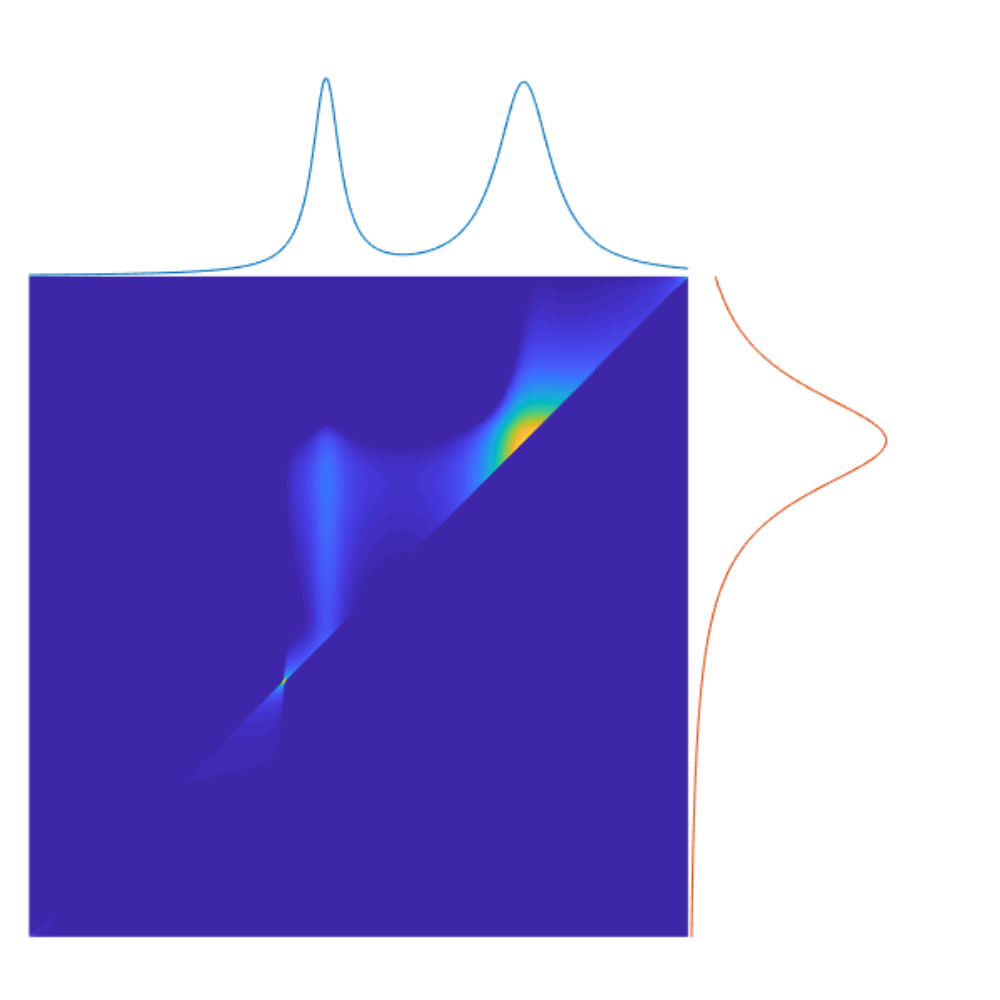}&
  \includegraphics[width=0.23\textwidth]{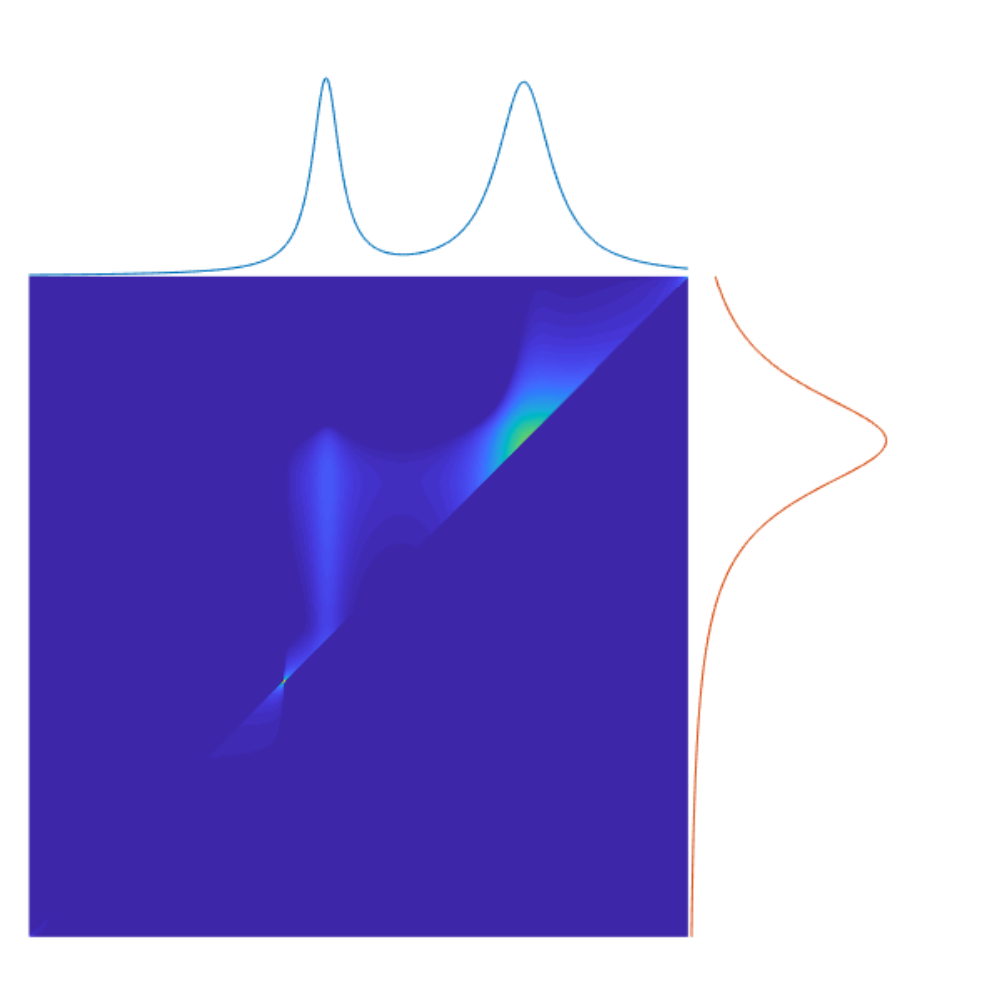}\\
                                                                        $\gamma = 1,000$ &
                                                                       $\gamma = 100$ &
                                                                        $\gamma = 10$&
                                                                                        $\gamma = 1$
  \end{tabular}
  \caption{Visualization of transport plans of the quadratically regularized optimal transport problem with $M=N=400$ and metric transport cost $c_{ij} =\abs{x_{i}-y_{j}}$.}
  \label{fig:qrot_gamma_to_zero_abs}
\end{figure}

\begin{figure}[htb]
  \centering
  \begin{tabular}{cccc}
  \includegraphics[width=0.23\textwidth]{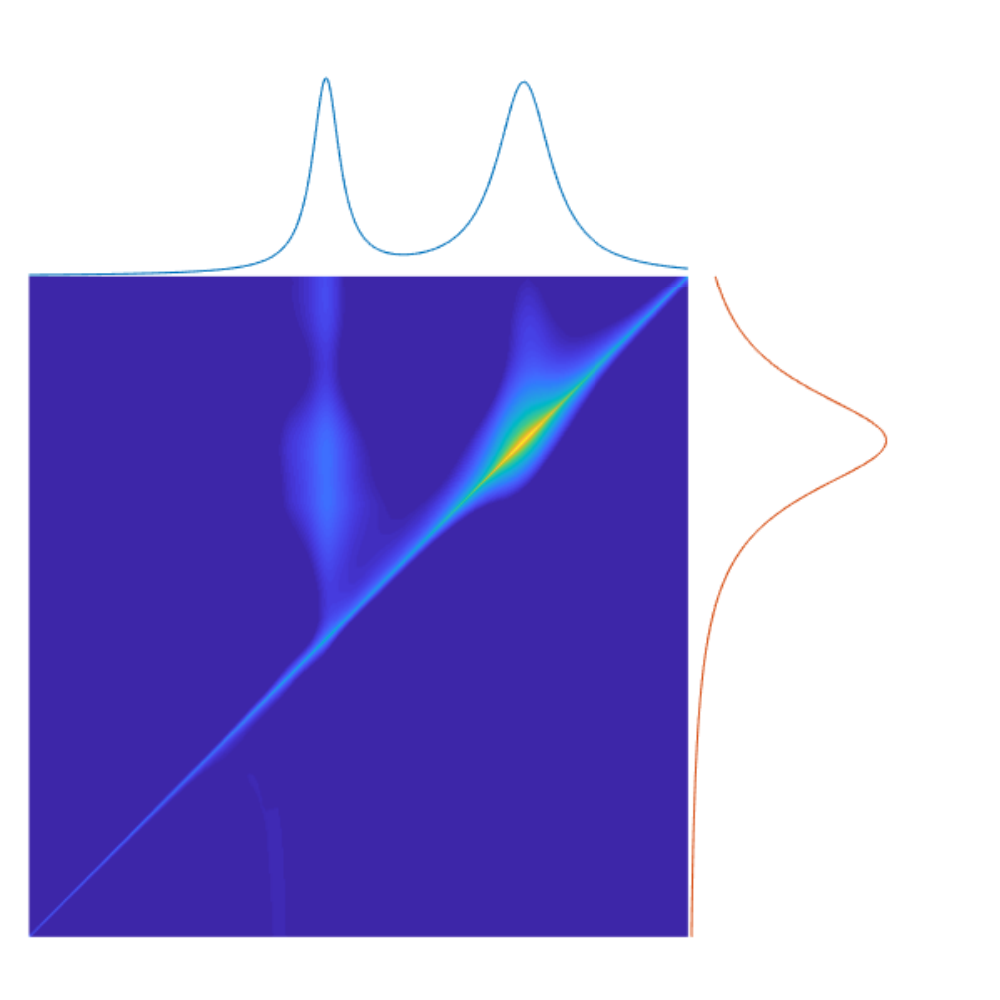}&
  \includegraphics[width=0.23\textwidth]{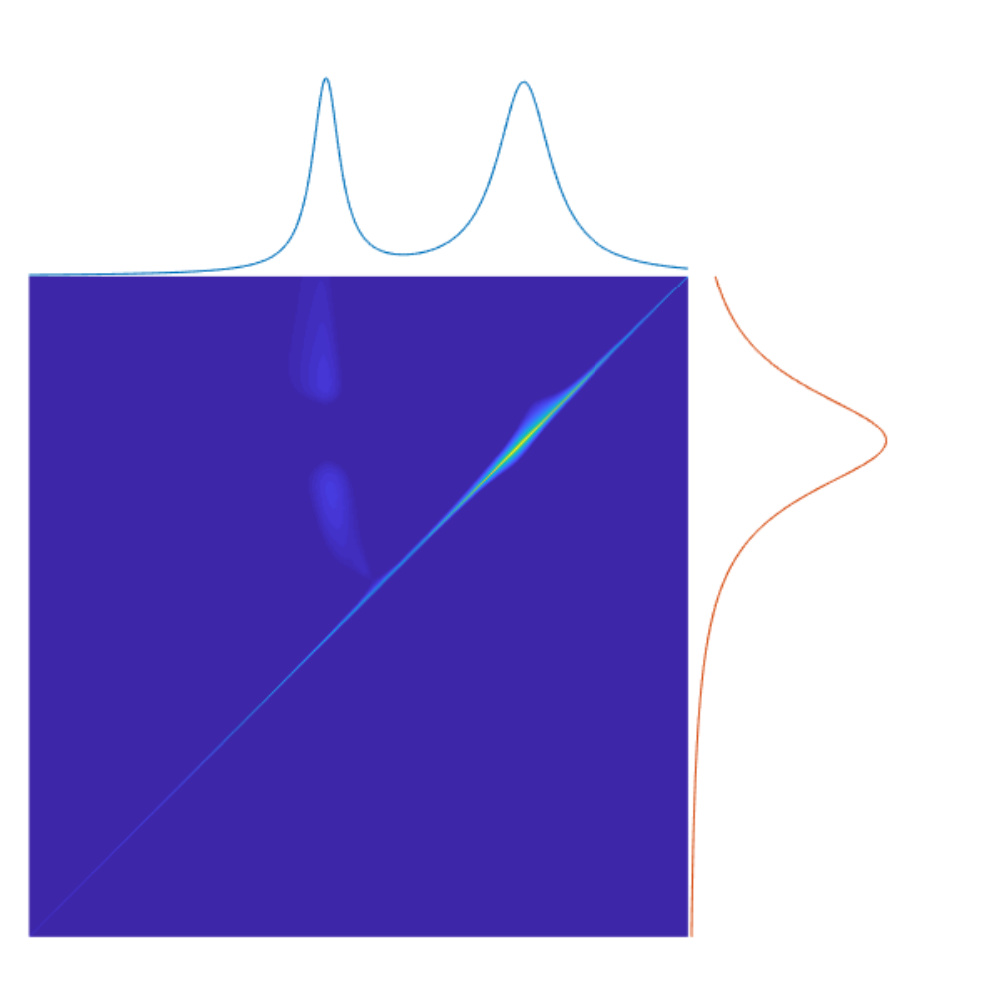}&
  \includegraphics[width=0.23\textwidth]{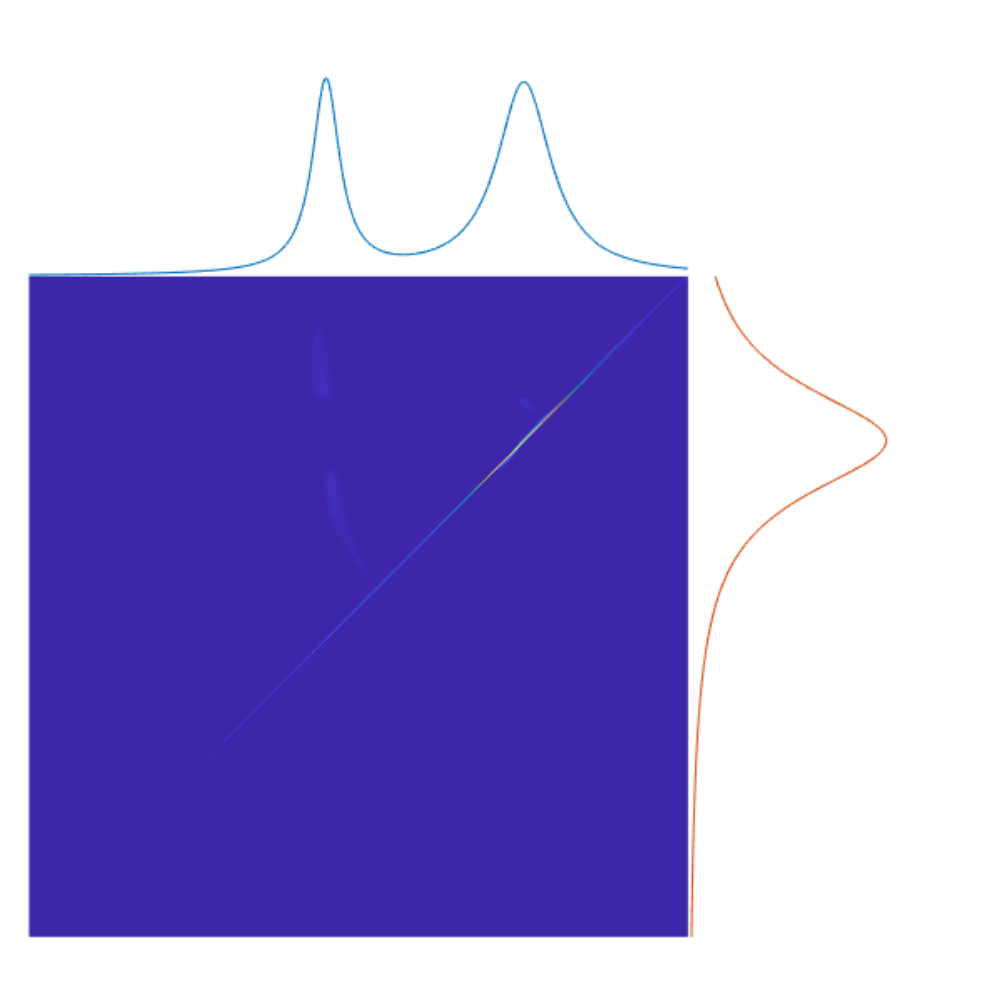}&
  \includegraphics[width=0.23\textwidth]{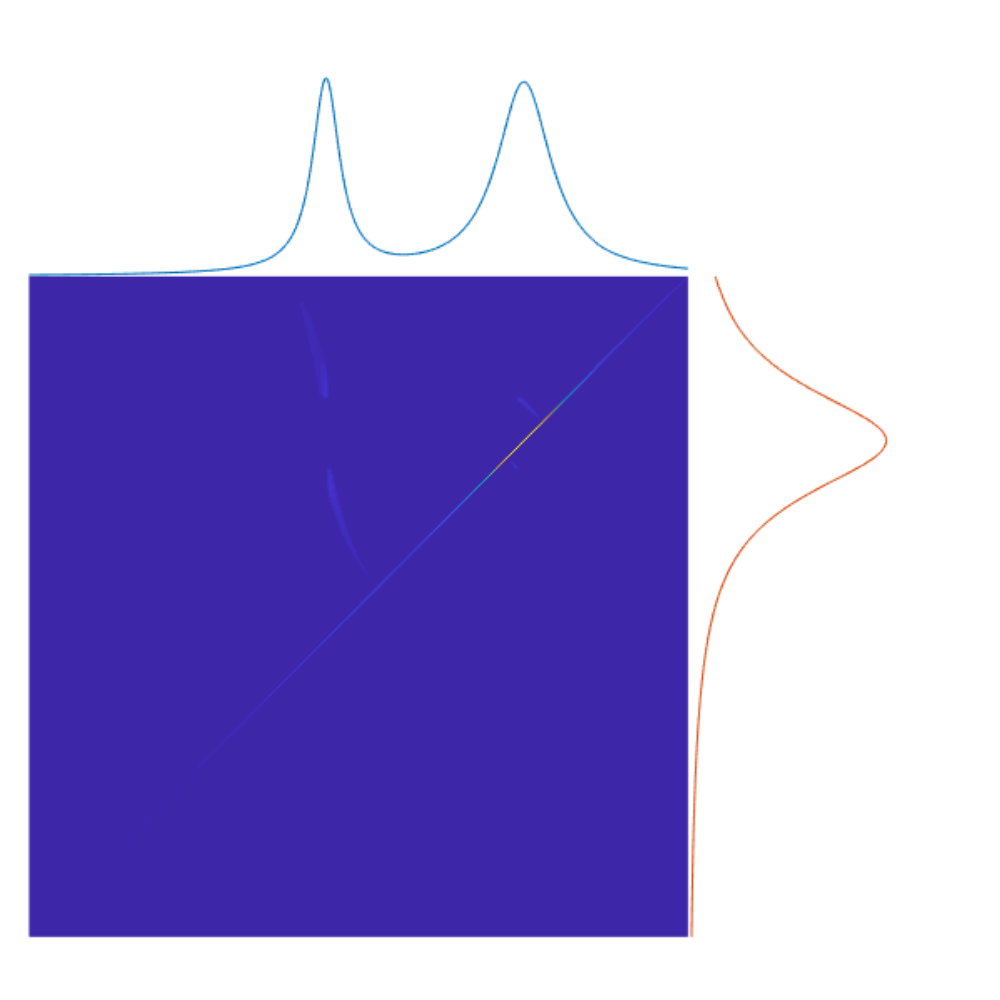}\\
                                                                        $\gamma = 1,000$ &
                                                                       $\gamma = 100$ &
                                                                        $\gamma = 10$&
                                                                                        $\gamma = 1$
  \end{tabular}
  \caption{Visualization of transport plans of the quadratically regularized optimal transport problem with $M=N=400$ and concave increasing transport cost $c_{ij} = \sqrt{\abs{x_{i}-y_{j}}}$.}
  \label{fig:qrot_gamma_to_zero_sqrt}
\end{figure}

\subsection{Mesh independence and comparsion of SSN and NLGS }
\label{sec:num-mesh-indep}

While we did not analyze our algorithms in the continuous case, we made an experiment to see how the methods converge when we change the mesh size of the discretization. To that end, we did a simple piecewise constant approximation of the marginals, the cost and the transport plan as described in Appendix~\ref{sec:discretization}.
This derivation shows that one has to scale up the marginals for finer discretization (or, equivalently, scale down the regularization parameter $\gamma$) to get consistent results.
We also took care to adapt the termination criteria so that we terminate the algorithms when the continuous counterpart of the termination criteria is satisfied (again, see Table~\ref{tbl:discretization_mapping} in Appendix~\ref{sec:discretization} for details). 

We used marginals $\mu^{\pm}:[0,1]\to[0,\infty[$ of the form
\[
\mu(x) = r\tfrac{1}{1+m(x-a)^{2}},,\qquad \nu(x) = s\Big(\tfrac{1}{1+m_{1}(x-a_{1})^{2}}+\tfrac{1}{1+m_{2}(x-a_{2})^{2}}\Big)
\]
with varying $m,m_{1},m_{2}>0,\ 0<a,a_{1},a_{2}<1$  and appropriate normalization factors $r,s$ and quadratic cost $c(x,y) = (x-y)^{2}$ and discretized each instance of the problem with $M=N$ varying from $10$ to $1,000$. We solved the problem for each size for regularization parameter $\gamma = 0.001$ with the semismooth Newton method from Algorithm~\ref{alg:qrot_ssn} (with parameters $\epsilon = 10^{-6}$ and Armijo parameters $\kappa = 0.5$ and $\theta = 0.1$) up to tolerance $10^{-3}$ and report the number of iterations needed in Figure~\ref{fig:mesh_independence_ssn}. As can be observed, the number of iterations is comparable for each instance of the problem. Moreover, it seems that the number of iterations does not grow with finer discretization (however, the number of iterations seems to oscillate unpredictable for coarse discretization). The would hint at mesh independence of the method and one could hope to prove this is future research. We performed a similar experiment for the nonlinear Gauss-Seidl method from Algorithm~\ref{alg:qrot_gs} (with larger regularization parameter $\gamma = 0.05$ and only up to $M=N=500$ and show the results in Figure~\ref{fig:mesh_independence_gs}. We see an overall increase of the number of iterations but only very slightly (with several instances where the number of iterations does not increasing with finer discretization).

\begin{figure}[htb]
  \centering
%
%
\definecolor{mycolor1}{rgb}{0.00000,0.44700,0.74100}%
\definecolor{mycolor2}{rgb}{0.85000,0.32500,0.09800}%
\definecolor{mycolor3}{rgb}{0.92900,0.69400,0.12500}%
\definecolor{mycolor4}{rgb}{0.49400,0.18400,0.55600}%
\definecolor{mycolor5}{rgb}{0.46600,0.67400,0.18800}%
\definecolor{mycolor6}{rgb}{0.30100,0.74500,0.93300}%
\definecolor{mycolor7}{rgb}{0.63500,0.07800,0.18400}%
\begin{tikzpicture}

\begin{axis}[%
width=6cm,
height=4cm,
scale only axis,
xmin=0,
xmax=1000,
xlabel style={font=\color{white!15!black}},
xlabel={$M=N$},
ymin=0,
ymax=100,
ylabel style={font=\color{white!15!black}},
ylabel={\# ssn iterations},
axis background/.style={fill=white}
]
\addplot [color=mycolor1, forget plot]
  table[row sep=crcr]{%
10	38\\
14	53\\
19	48\\
27	67\\
37	69\\
52	75\\
72	50\\
100	45\\
139	34\\
193	26\\
268	29\\
373	29\\
518	32\\
720	35\\
1000	35\\
};
\addplot [color=mycolor2, forget plot]
  table[row sep=crcr]{%
10	34\\
14	36\\
19	28\\
27	29\\
37	25\\
52	34\\
72	32\\
100	22\\
139	18\\
193	17\\
268	22\\
373	20\\
518	25\\
720	26\\
1000	28\\
};
\addplot [color=mycolor3, forget plot]
  table[row sep=crcr]{%
10	37\\
14	47\\
19	51\\
27	61\\
37	59\\
52	68\\
72	52\\
100	42\\
139	29\\
193	28\\
268	27\\
373	29\\
518	28\\
720	29\\
1000	30\\
};
\addplot [color=mycolor4, forget plot]
  table[row sep=crcr]{%
10	45\\
14	48\\
19	47\\
27	42\\
37	33\\
52	31\\
72	32\\
100	29\\
139	24\\
193	25\\
268	31\\
373	28\\
518	37\\
720	40\\
1000	42\\
};
\addplot [color=mycolor5, forget plot]
  table[row sep=crcr]{%
10	30\\
14	32\\
19	26\\
27	33\\
37	27\\
52	22\\
72	24\\
100	21\\
139	18\\
193	19\\
268	21\\
373	20\\
518	18\\
720	18\\
1000	18\\
};
\addplot [color=mycolor6, forget plot]
  table[row sep=crcr]{%
10	48\\
14	63\\
19	63\\
27	74\\
37	78\\
52	51\\
72	44\\
100	35\\
139	25\\
193	23\\
268	31\\
373	25\\
518	36\\
720	38\\
1000	40\\
};
\addplot [color=mycolor7, forget plot]
  table[row sep=crcr]{%
10	30\\
14	35\\
19	38\\
27	39\\
37	51\\
52	33\\
72	32\\
100	30\\
139	21\\
193	24\\
268	20\\
373	22\\
518	24\\
720	23\\
1000	23\\
};
\addplot [color=mycolor1, forget plot]
  table[row sep=crcr]{%
10	24\\
14	33\\
19	28\\
27	28\\
37	15\\
52	19\\
72	19\\
100	17\\
139	18\\
193	18\\
268	19\\
373	20\\
518	21\\
720	22\\
1000	22\\
};
\addplot [color=mycolor2, forget plot]
  table[row sep=crcr]{%
10	34\\
14	31\\
19	31\\
27	23\\
37	27\\
52	36\\
72	26\\
100	25\\
139	19\\
193	18\\
268	23\\
373	20\\
518	25\\
720	27\\
1000	28\\
};
\addplot [color=mycolor3, forget plot]
  table[row sep=crcr]{%
10	20\\
14	32\\
19	22\\
27	29\\
37	25\\
52	22\\
72	28\\
100	20\\
139	20\\
193	21\\
268	22\\
373	24\\
518	26\\
720	29\\
1000	30\\
};
\addplot [color=mycolor4, forget plot]
  table[row sep=crcr]{%
10	36\\
14	50\\
19	65\\
27	62\\
37	45\\
52	43\\
72	38\\
100	30\\
139	25\\
193	22\\
268	28\\
373	25\\
518	31\\
720	33\\
1000	35\\
};
\addplot [color=mycolor5, forget plot]
  table[row sep=crcr]{%
10	13\\
14	35\\
19	32\\
27	33\\
37	43\\
52	32\\
72	22\\
100	21\\
139	23\\
193	23\\
268	21\\
373	26\\
518	23\\
720	24\\
1000	24\\
};
\addplot [color=mycolor6, forget plot]
  table[row sep=crcr]{%
10	30\\
14	33\\
19	49\\
27	58\\
37	59\\
52	57\\
72	35\\
100	27\\
139	22\\
193	19\\
268	21\\
373	21\\
518	23\\
720	23\\
1000	25\\
};
\addplot [color=mycolor7, forget plot]
  table[row sep=crcr]{%
10	31\\
14	34\\
19	34\\
27	40\\
37	45\\
52	34\\
72	39\\
100	23\\
139	22\\
193	19\\
268	17\\
373	21\\
518	17\\
720	20\\
1000	19\\
};
\addplot [color=mycolor1, forget plot]
  table[row sep=crcr]{%
10	55\\
14	67\\
19	76\\
27	72\\
37	52\\
52	51\\
72	49\\
100	32\\
139	24\\
193	24\\
268	33\\
373	26\\
518	36\\
720	39\\
1000	40\\
};
\addplot [color=mycolor2, forget plot]
  table[row sep=crcr]{%
10	29\\
14	41\\
19	36\\
27	39\\
37	41\\
52	38\\
72	32\\
100	37\\
139	26\\
193	21\\
268	21\\
373	23\\
518	27\\
720	27\\
1000	27\\
};
\addplot [color=mycolor3, forget plot]
  table[row sep=crcr]{%
10	39\\
14	55\\
19	36\\
27	43\\
37	36\\
52	32\\
72	31\\
100	25\\
139	21\\
193	22\\
268	29\\
373	24\\
518	32\\
720	34\\
1000	36\\
};
\addplot [color=mycolor4, forget plot]
  table[row sep=crcr]{%
10	36\\
14	35\\
19	35\\
27	36\\
37	37\\
52	37\\
72	29\\
100	24\\
139	20\\
193	21\\
268	22\\
373	23\\
518	23\\
720	26\\
1000	25\\
};
\addplot [color=mycolor5, forget plot]
  table[row sep=crcr]{%
10	41\\
14	54\\
19	55\\
27	75\\
37	97\\
52	87\\
72	71\\
100	49\\
139	45\\
193	37\\
268	28\\
373	34\\
518	30\\
720	32\\
1000	33\\
};
\addplot [color=mycolor6, forget plot]
  table[row sep=crcr]{%
10	37\\
14	47\\
19	63\\
27	64\\
37	87\\
52	68\\
72	53\\
100	47\\
139	38\\
193	31\\
268	24\\
373	34\\
518	28\\
720	28\\
1000	29\\
};
\addplot [color=mycolor7, forget plot]
  table[row sep=crcr]{%
10	35\\
14	36\\
19	30\\
27	33\\
37	29\\
52	26\\
72	22\\
100	20\\
139	20\\
193	18\\
268	19\\
373	20\\
518	21\\
720	22\\
1000	23\\
};
\addplot [color=mycolor1, forget plot]
  table[row sep=crcr]{%
10	22\\
14	35\\
19	36\\
27	28\\
37	28\\
52	39\\
72	30\\
100	23\\
139	17\\
193	18\\
268	20\\
373	20\\
518	19\\
720	26\\
1000	27\\
};
\addplot [color=mycolor2, forget plot]
  table[row sep=crcr]{%
10	53\\
14	54\\
19	57\\
27	54\\
37	41\\
52	39\\
72	37\\
100	30\\
139	21\\
193	22\\
268	29\\
373	24\\
518	32\\
720	34\\
1000	37\\
};
\addplot [color=mycolor3, forget plot]
  table[row sep=crcr]{%
10	33\\
14	44\\
19	42\\
27	38\\
37	54\\
52	52\\
72	52\\
100	35\\
139	24\\
193	23\\
268	24\\
373	25\\
518	26\\
720	26\\
1000	28\\
};
\addplot [color=mycolor4, forget plot]
  table[row sep=crcr]{%
10	40\\
14	44\\
19	43\\
27	43\\
37	39\\
52	33\\
72	27\\
100	24\\
139	21\\
193	22\\
268	27\\
373	25\\
518	30\\
720	31\\
1000	34\\
};
\end{axis}
\end{tikzpicture}%
  \caption{Number of iteration for the semismooth Newton method to achieve a desired accuracy. Each graph corresponds to one instance of the problem.}\label{fig:mesh_independence_ssn}
\end{figure}
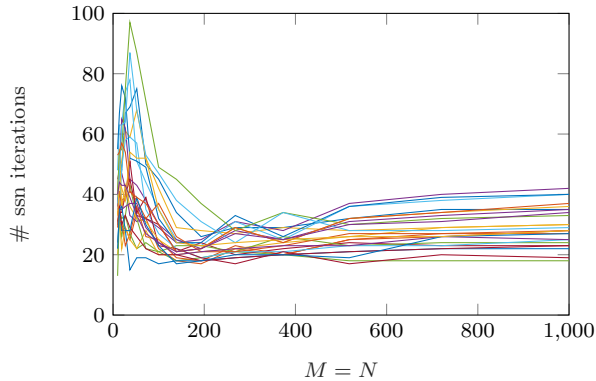

\begin{figure}[htb]
  \centering
%
%
\definecolor{mycolor1}{rgb}{0.00000,0.44700,0.74100}%
\definecolor{mycolor2}{rgb}{0.85000,0.32500,0.09800}%
\definecolor{mycolor3}{rgb}{0.92900,0.69400,0.12500}%
\definecolor{mycolor4}{rgb}{0.49400,0.18400,0.55600}%
\definecolor{mycolor5}{rgb}{0.46600,0.67400,0.18800}%
\definecolor{mycolor6}{rgb}{0.30100,0.74500,0.93300}%
\definecolor{mycolor7}{rgb}{0.63500,0.07800,0.18400}%
\begin{tikzpicture}

\begin{axis}[%
width=6cm,
height=4cm,
scale only axis,
xmin=0,
xmax=500,
xlabel style={font=\color{white!15!black}},
xlabel={$M=N$},
ymin=0,
ymax=160,
ylabel style={font=\color{white!15!black}},
ylabel={\# gs iterations},
axis background/.style={fill=white}
]
\addplot [color=mycolor1, forget plot]
  table[row sep=crcr]{%
10	19\\
13	16\\
17	23\\
23	20\\
31	22\\
40	25\\
53	25\\
71	27\\
94	27\\
124	29\\
164	28\\
216	30\\
286	33\\
378	36\\
500	40\\
};
\addplot [color=mycolor2, forget plot]
  table[row sep=crcr]{%
10	30\\
13	31\\
17	39\\
23	44\\
31	51\\
40	52\\
53	53\\
71	50\\
94	46\\
124	39\\
164	36\\
216	40\\
286	44\\
378	49\\
500	55\\
};
\addplot [color=mycolor3, forget plot]
  table[row sep=crcr]{%
10	43\\
13	42\\
17	47\\
23	50\\
31	56\\
40	63\\
53	70\\
71	74\\
94	75\\
124	79\\
164	80\\
216	49\\
286	44\\
378	37\\
500	29\\
};
\addplot [color=mycolor4, forget plot]
  table[row sep=crcr]{%
10	31\\
13	34\\
17	37\\
23	47\\
31	53\\
40	59\\
53	58\\
71	56\\
94	50\\
124	47\\
164	50\\
216	53\\
286	57\\
378	61\\
500	65\\
};
\addplot [color=mycolor5, forget plot]
  table[row sep=crcr]{%
10	30\\
13	42\\
17	45\\
23	52\\
31	61\\
40	65\\
53	66\\
71	66\\
94	67\\
124	64\\
164	57\\
216	56\\
286	56\\
378	60\\
500	65\\
};
\addplot [color=mycolor6, forget plot]
  table[row sep=crcr]{%
10	23\\
13	21\\
17	27\\
23	32\\
31	35\\
40	39\\
53	43\\
71	49\\
94	53\\
124	58\\
164	62\\
216	66\\
286	69\\
378	72\\
500	74\\
};
\addplot [color=mycolor7, forget plot]
  table[row sep=crcr]{%
10	26\\
13	19\\
17	20\\
23	19\\
31	18\\
40	21\\
53	20\\
71	22\\
94	24\\
124	27\\
164	30\\
216	34\\
286	38\\
378	42\\
500	46\\
};
\addplot [color=mycolor1, forget plot]
  table[row sep=crcr]{%
10	11\\
13	12\\
17	13\\
23	15\\
31	15\\
40	16\\
53	17\\
71	18\\
94	19\\
124	20\\
164	20\\
216	21\\
286	21\\
378	20\\
500	20\\
};
\addplot [color=mycolor2, forget plot]
  table[row sep=crcr]{%
10	29\\
13	29\\
17	32\\
23	39\\
31	46\\
40	52\\
53	62\\
71	72\\
94	83\\
124	95\\
164	106\\
216	114\\
286	120\\
378	129\\
500	141\\
};
\addplot [color=mycolor3, forget plot]
  table[row sep=crcr]{%
10	38\\
13	46\\
17	52\\
23	52\\
31	51\\
40	55\\
53	50\\
71	40\\
94	40\\
124	38\\
164	37\\
216	36\\
286	36\\
378	35\\
500	34\\
};
\addplot [color=mycolor4, forget plot]
  table[row sep=crcr]{%
10	28\\
13	31\\
17	37\\
23	40\\
31	48\\
40	55\\
53	65\\
71	75\\
94	84\\
124	91\\
164	95\\
216	99\\
286	104\\
378	109\\
500	116\\
};
\addplot [color=mycolor5, forget plot]
  table[row sep=crcr]{%
10	19\\
13	26\\
17	30\\
23	34\\
31	37\\
40	40\\
53	42\\
71	42\\
94	44\\
124	43\\
164	39\\
216	41\\
286	41\\
378	42\\
500	42\\
};
\addplot [color=mycolor6, forget plot]
  table[row sep=crcr]{%
10	32\\
13	35\\
17	45\\
23	48\\
31	55\\
40	62\\
53	73\\
71	85\\
94	94\\
124	105\\
164	111\\
216	108\\
286	105\\
378	100\\
500	96\\
};
\addplot [color=mycolor7, forget plot]
  table[row sep=crcr]{%
10	23\\
13	26\\
17	33\\
23	38\\
31	43\\
40	48\\
53	53\\
71	58\\
94	61\\
124	63\\
164	51\\
216	53\\
286	58\\
378	66\\
500	75\\
};
\addplot [color=mycolor1, forget plot]
  table[row sep=crcr]{%
10	43\\
13	42\\
17	49\\
23	52\\
31	56\\
40	62\\
53	69\\
71	76\\
94	84\\
124	94\\
164	104\\
216	114\\
286	122\\
378	128\\
500	138\\
};
\addplot [color=mycolor2, forget plot]
  table[row sep=crcr]{%
10	22\\
13	27\\
17	36\\
23	33\\
31	35\\
40	37\\
53	33\\
71	37\\
94	40\\
124	43\\
164	47\\
216	52\\
286	58\\
378	64\\
500	70\\
};
\addplot [color=mycolor3, forget plot]
  table[row sep=crcr]{%
10	20\\
13	21\\
17	27\\
23	28\\
31	33\\
40	37\\
53	41\\
71	42\\
94	45\\
124	49\\
164	51\\
216	52\\
286	55\\
378	57\\
500	59\\
};
\addplot [color=mycolor4, forget plot]
  table[row sep=crcr]{%
10	23\\
13	26\\
17	28\\
23	26\\
31	20\\
40	21\\
53	22\\
71	23\\
94	23\\
124	25\\
164	25\\
216	24\\
286	23\\
378	19\\
500	20\\
};
\addplot [color=mycolor5, forget plot]
  table[row sep=crcr]{%
10	54\\
13	63\\
17	70\\
23	76\\
31	77\\
40	81\\
53	84\\
71	87\\
94	90\\
124	88\\
164	49\\
216	43\\
286	42\\
378	42\\
500	42\\
};
\addplot [color=mycolor6, forget plot]
  table[row sep=crcr]{%
10	33\\
13	41\\
17	49\\
23	55\\
31	61\\
40	61\\
53	60\\
71	57\\
94	54\\
124	51\\
164	55\\
216	59\\
286	63\\
378	68\\
500	73\\
};
\addplot [color=mycolor7, forget plot]
  table[row sep=crcr]{%
10	34\\
13	37\\
17	38\\
23	49\\
31	50\\
40	56\\
53	64\\
71	72\\
94	82\\
124	94\\
164	105\\
216	117\\
286	130\\
378	144\\
500	157\\
};
\addplot [color=mycolor1, forget plot]
  table[row sep=crcr]{%
10	29\\
13	24\\
17	28\\
23	28\\
31	29\\
40	32\\
53	30\\
71	35\\
94	38\\
124	42\\
164	48\\
216	53\\
286	59\\
378	65\\
500	72\\
};
\addplot [color=mycolor2, forget plot]
  table[row sep=crcr]{%
10	31\\
13	23\\
17	27\\
23	22\\
31	22\\
40	23\\
53	23\\
71	23\\
94	25\\
124	27\\
164	30\\
216	32\\
286	33\\
378	36\\
500	40\\
};
\addplot [color=mycolor3, forget plot]
  table[row sep=crcr]{%
10	38\\
13	47\\
17	51\\
23	63\\
31	65\\
40	74\\
53	79\\
71	86\\
94	92\\
124	102\\
164	110\\
216	117\\
286	124\\
378	135\\
500	144\\
};
\addplot [color=mycolor4, forget plot]
  table[row sep=crcr]{%
10	43\\
13	44\\
17	51\\
23	49\\
31	53\\
40	53\\
53	54\\
71	41\\
94	36\\
124	34\\
164	31\\
216	30\\
286	30\\
378	31\\
500	32\\
};
\end{axis}
\end{tikzpicture}%
  \caption{Number of iteration for the nonlinear Gauss-Seidel semismooth method to achieve a desired accuracy. Each graph corresponds to one instance of the problem.}\label{fig:mesh_independence_gs}
\end{figure}
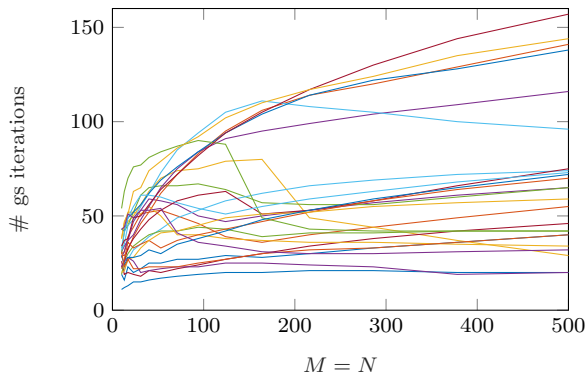

\subsection{Optimal transport between empirical distributions}
\label{sec:empiricial}

As an example in two space dimensions, we consider two distributions $\mu,\nu$. Instead of using these as marginals, we consider empirical distributions, i.e. we generate samples $(x_{i})_{i=1,\dots,N}$, sampled from $\mu$ and $(y_{j})_{j=1,\dots,M}$, sampled from $\nu$. These samples give empirical approximations
\[
\hat\mu = \tfrac{1}{N}\sum_{i=1}^{N}\delta_{x_{i}},\quad\hat\nu = \tfrac{1}{M}\sum_{j=1}^{M}\delta_{y_{j}}.
\]
The optimal transport problem~\eqref{qrot-cont} with these two marginals does no fulfill Assumption 1, since the marginals are not $L^{2}$-functions. However, we can consider it as a discrete problem optimal transport problem in the form~\eqref{eq:qrot-discrete} when we denote $c_{ij} = c(x_{i},y_{j})$ (for some cost $c$) and marginals $\one_{M}$ and $\one_{N}$, respectively.
We solve this discrete optimal transport problem and obtain a transport plan $\pi^{*}$. Since we use quadratic regularization, the plan will be sparse and hence, we can visualize it by plotting arrows from $x_{i}$ to $y_{j}$ and we make the thickness of the arrows proportional to the size of the entry $\pi^{*}_{ij}$. In other words: The thickness of the arrow from $x_{i}$ to $y_{j}$ indicates how much of the mass in $x_{i}$ has been transported to $y_{j}$.
In Figure~\ref{fig:empirical_q} we show the result for $N=80$ samples from an anisotropic Gaussian distribution (centered at the origin) and $M=120$ samples from a uniform distribution on a segment of an annulus. We used $c(x_{i},y_{j}) = \norm{x_{i}-y_{j}}^{2}$ with the Euclidean norm and regularization paramater $\gamma=1$. The resulting plan $\pi^{*}$ has $212$ non-zero entries.
For a comparison we show the result of entropically regularized optimal transport in the same situation in Figure~\ref{fig:empirical_s}. We used $\gamma= 0.05$ (which is the smallest value for which our naive implementation of Sinkhorn algorithms is still stable). The resulting plan has $6730$ nonzero entries and we only plot lines for the transport which are larger than 1\% of the largest entry in the optimal transport plan.

\begin{figure}[htb]
  \centering
  \includegraphics[width=\textwidth]{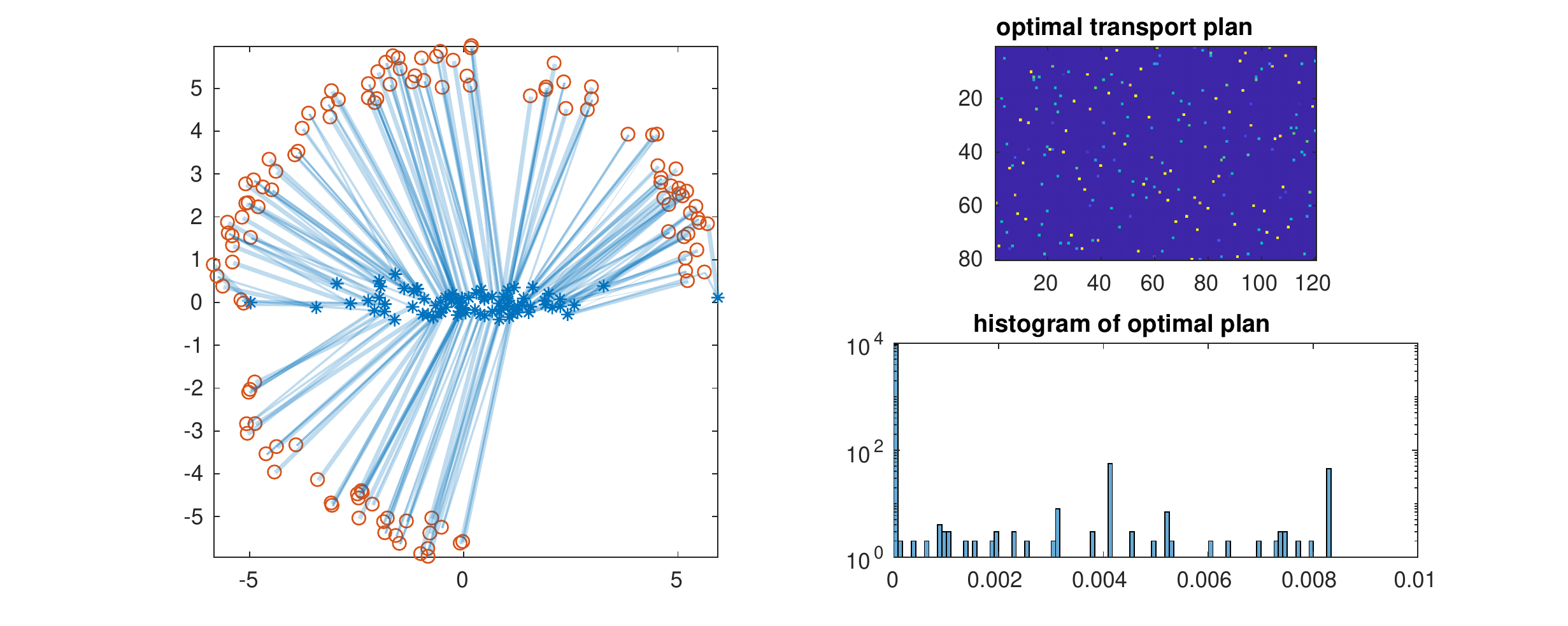}
  \caption{Illustration of the quadratically regularized optimal transport between empirical distributions. Left: Source distribution $\hat\mu$ denoted by blue starts and target distribution $\hat \nu$ denoted by red circles together with lines that indicate the transport. Right: The transport plan and its histogram in semi-log scale.}\label{fig:empirical_q}
\end{figure}
\begin{figure}[htb]
  \centering
  \includegraphics[width=\textwidth]{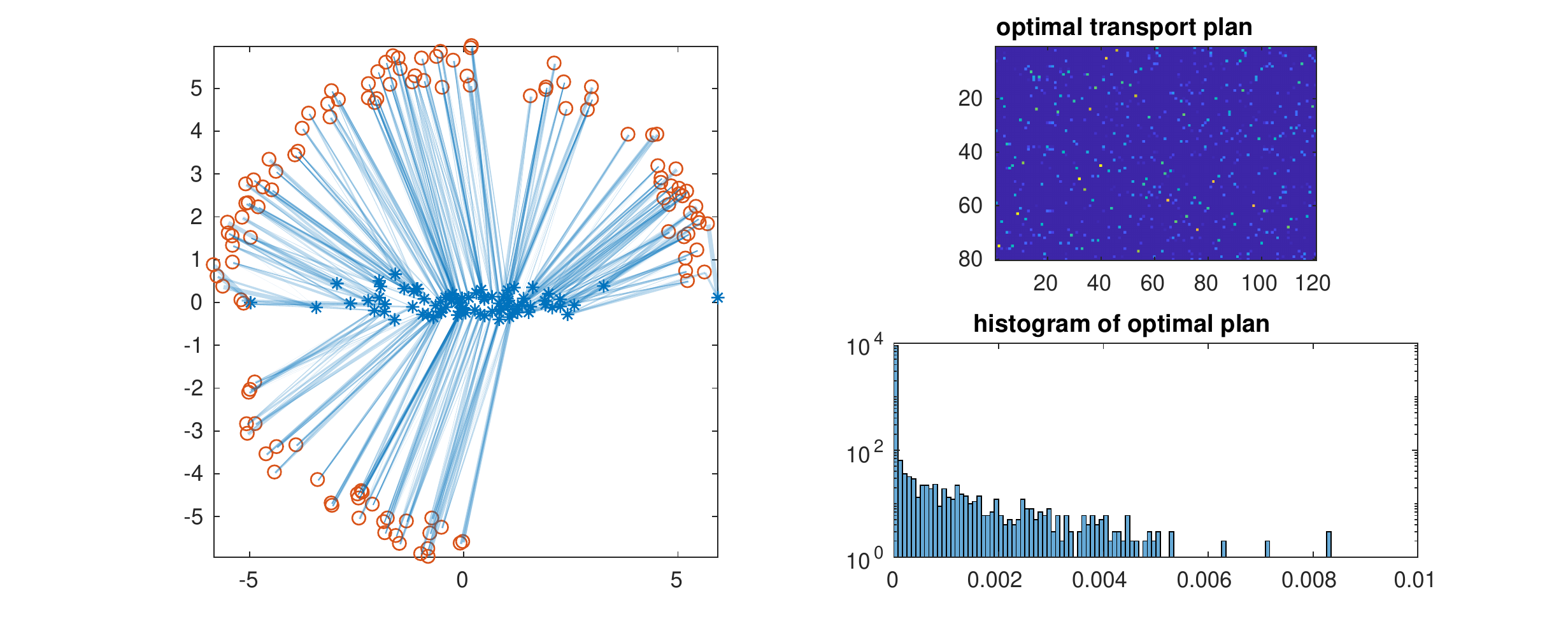}
  \caption{Illustration of the entropically regularized optimal transport between empirical distributions. Left: Source distribution $\hat\mu$ denoted by blue starts and target distribution $\hat \nu$ denoted by red circles together with lines that indicate the transport. Right: The transport plan and its histogram in semi-log scale.}\label{fig:empirical_s}
\end{figure}

\section{Conclusion}
\label{sec:conclusion}
We analyzed the quadratically regularized optimal transport problem in Kantorovich form.
While it is straight forward to derive the dual problem, our proof of existence of dual optima is quite intricate. We note that we are not aware of any proof of existence of the dual of other regularized transport problems in the continuous case besides the very recent~\cite{clason2019entropic} for entropic regularization.
We derived two algorithms to solve the dual problems, both of which converge by standard results.
It turns out that the semismooth quasi-Newton methods converges fast in all cases and that it behaves stably with respect to the regularization parameter in our numerical experiments. We even observe mesh independence of the method in the experiments.
One drawback of the semismooth Newton method is (compared with, e.g.\ the Sinkhorn iteration~\cite{cuturi2013sinkhorn}), is that we need to assemble the Newton matrix in each step. While this matrix is usually very sparse, one still needs to check $MN$ cases, which may be too large for large scale problems.
We did not investigate, how special structure of the cost function $c$ may help to reduce the cost to assemble the sparse matrix $\sigma$.

\begin{acknowledgements}
  We would like to thank the reviewer for helpful suggestions that lead to an improvement presentation  
  and also Stephan Walther (TU Dortmund) for helping with the construction of the counterexample in Section~\ref{sec:reg-dual}.
\end{acknowledgements}

\appendix
\section{Discretization with piecewise-constant ansatz functions}
\label{sec:discretization}
For sake of brevity, we just consider an equidistant discretization of $[0,1]$ into $N$ intervals using piecewise constant ansatz functions, i.e.
\begin{align*}
	\pi(x,y) &\coloneqq \sum_{i,j=0}^{N-1} \pi_{ij} \chi_{(\frac{i}{N},\frac{i+1}{N})\times(\frac{j}{N},\frac{j+1}{N})}(x,y), 
\end{align*}
for coefficients $\pi_{ij}$ and assume analogous definitions for the quantities $c$, $\mu^+$, $\mu^-$, $\alpha$ and $\beta$. They have to coincide on average
over the intervals. Again, we study this for $\pi$ and obtain that the identity
\begin{align*}
	\int_{\frac{i}{N}}^{\frac{i+1}{N}} \int_{\frac{j}{N}}^{\frac{j+1}{N}} \pi(x,y) \,\dd y \dd x
	&= \int_{\frac{i}{N}}^{\frac{i+1}{N}} \int_{\frac{j}{N}}^{\frac{j+1}{N}}  
	   \sum_{i,j=0}^{N-1} \pi_{ij} \chi_{(\frac{i}{N},\frac{i+1}{N})\times(\frac{j}{N},\frac{j+1}{N})}(x,y) \,\dd y \dd x  \\
	&= \frac{1}{N^2} \pi_{ij}
\end{align*}
holds. Again, analogous identities hold for the quantities $c$, $\mu^+$, $\mu^-$, $\alpha$ and $\beta$. The ones with one-dimensional domain
are scaled by $\frac{1}{N}$ instead of $\frac{1}{N^2}$. 

Now, we consider the discrete Algorithm~\ref{alg:qrot_ssn}, which operates on discrete quantities and establish a consistent mapping 
of the quantities from the discretization to the ones of the solver. 
We denote its input quantities by $\bar{c}_{ij}$, $\bar{\mu}^-_i$, $\bar{\mu}^+_i$ and its output quantities by $\bar{\alpha}_i$, $\bar{\beta}_j$, $\bar{pi}_{ij}$,  and $\bar{E}$. It solves for 
\begin{align*}
	\sum_{j=0}^{N-1} \bar{\pi}_{ij} = \gamma \bar{\mu}_i^+,
\end{align*}
which we desire to correspond to
\begin{align*}
	\int_{\frac{i}{N}}^{\frac{i+1}{N}} \int_0^1 \pi(x,y)\,\dd y \dd x = \int_{\frac{i}{N}}^{\frac{i+1}{N}}\mu^+(x) \,\dd x.
\end{align*}
We plug in the ansatz functions and obtain the identity
\begin{align*}
	\frac{1}{N^2} \sum_{j=0}^{N-1} \pi_{ij} = \frac{1}{N} \mu^+_i .
\end{align*}
We set $\bar{\pi}_{ij} \coloneqq \gamma \pi_{ij}$ and obtain
\begin{align*}
	\sum_{j=0}^{N-1} \bar{\pi}_{ij} = \sum_{j=0}^{N-1} \pi_{ij} = N \mu^-_i.
\end{align*}
Thus, the choice $\bar{\mu}^-_i \coloneqq N \mu^-_i$ gives a consistent conversion. Similarly, we obtain $\bar{\mu}^+_j \coloneqq N \mu^+_j$.
We proceed with the objective. Plugging in the ansatz functions into the continuous objective gives
\begin{align*}
	E = \frac{\gamma}{2}  \frac{1}{N^2} \sum_{i,j=0}^{N-1} \pi_{ij}^2 - \frac{1}{N} \left(\sum_{i=0}^{N-1} \alpha_i \mu^-_i + \sum_{j=0}^{N-1} \beta_j \mu^+_j \right).
\end{align*}
The solver computes
\begin{align*}
	\bar{E} = \frac{1}{2}  \sum_{i,j=0}^{N-1} \bar{\pi}_{ij}^2 - \gamma \left(\sum_{i=0}^{N-1} \bar{\alpha}_i \bar{\mu}^-_i + \sum_{j=0}^{N-1} \bar{\beta}_j \bar{\mu}^+_j \right),
\end{align*}
Plugging in $N \mu^-_i = \bar{\mu}^-_i$, $N \mu^+_j = \bar{\mu}^+_j$ and $\gamma \pi_{ij} = \bar{\pi}_{ij}$ gives
\begin{align*}
	\bar{E} = \gamma^2 \frac{1}{2}  \sum_{i,j=0}^{N-1} \pi_{ij}^2 - \gamma N \left(\sum_{i=0}^{N-1} \bar{\alpha}_i \mu^-_i + \sum_{j=0}^{N-1} \bar{\beta}_j \mu^+_j \right).
\end{align*}
Thus, the consistent identity $E = \frac{1}{\gamma N^2} \bar{E}$ follows if we choose $\bar{\alpha}_i \coloneqq \alpha_i$ and $\bar{\beta}_i \coloneqq \beta_i$.
The solver computes $\bar{\alpha}_i$ as the solution of
\begin{align*}
	\sum_{j=0}^{N-1}(\bar{\alpha}_i + \bar{\beta}_j - \bar{c}_{ij})_+ = \gamma \bar{\mu}_i^-,
\end{align*}
whereas the discretization of the corresponding continuous equation reads
\begin{align*}
	\frac{1}{N} \sum_{j=0}^{N-1} (\alpha_i + \beta_j  - c_{ij})_+ = \gamma \mu^-_i
\end{align*}
in terms of the coefficients. Plugging in the choices $\alpha_i = \bar{\alpha}_i$, $\beta_j = \bar{\beta}_j$, $c_{ij} = \bar{c}_{ij}$ and $N \mu^-_i = \bar{\mu}^-_i$
yields equivalence of the latter equation to
\begin{align*}
	\frac{1}{N} \sum_{j=0}^{N-1} (\bar{\alpha}_i + \bar{\beta}_j  - \bar{c}_{ij})_+  = \gamma \frac{1}{N} \bar{\mu}^-_j,
\end{align*}
which is equivalent to the equation that is solved by Algorithm~\ref{alg:qrot_ssn}. The argument for $\bar{\mu}^+_j$ is carried out analogously.

Regarding termination, the solver checks the criteria
\begin{gather*}
	\left|\frac{1}{\gamma}\sum_{j=0}^{M-1} \bar{\pi}_{ij} - \bar{\mu}_i^-\right| < \tau\quad\text{ and }\quad 
	\left|\frac{1}{\gamma}\sum_{i=0}^{M-1} \bar{\pi}_{ij} - \bar{\mu}_j^+\right| < \tau.
\end{gather*}
We only consider the first and plug the identity $\gamma \pi_{ij} = \bar{\pi}_{ij}$ into it, which gives equivalence to
\begin{align*}
	\left|\sum_{j=0}^{M-1} \pi_{ij} - N \mu_i^-\right| &< \tau. 
\end{align*}
This in turn is equivalent to
\begin{align*}
	N^2 \left|\sum_{j=0}^{M-1} \int_{\frac{i}{N}}^{\frac{i+1}{N}} \int_{\frac{j}{N}}^{\frac{j+1}{N}} \pi(x,y) \,\dd y \dd x 
	- \int_{\frac{i}{N}}^{\frac{i+1}{N}} \mu^-(x) \,\dd x\right| &< \tau, \\
	\left|\int_{\frac{i}{N}}^{\frac{i+1}{N}} \int_{0}^1 \pi(x,y) \,\dd y - \mu^-(x) \,\dd x\right| &< \frac{\tau}{N^2}.
\end{align*}

Moreover, the ansatz functions for $\pi$ and $\mu^-$ are constant on $\left(\frac{i}{N},\frac{i+1}{N}\right)$, which induces equivalence to
\begin{align*}
	\int_{\frac{i}{N}}^{\frac{i+1}{N}} \left|\int_{0}^1 \pi(x,y) \,\dd y  - \mu^-(x) \right| \,\dd x &< \frac{\tau}{N^2}.
\end{align*}
This implies that if the solver terminates, we have
\begin{align*}
	\left\| \int_0^1 \pi(\cdot, y)\,\dd y - \mu^-(\cdot) \right\|_{L^1((0,1))} < \frac{\tau}{N}.
\end{align*}
We summarize the choices for the consistent mapping of quantities arising from the discretization to quantities the solver operates on in 
Table\ \ref{tbl:discretization_mapping}.
\begin{table}[htb]
\caption{Mapping discretization quantities to solver quantities.}\label{tbl:discretization_mapping}
\begin{tabular}{ccc}
	\hline
	Coefficient & Solver Quantity & Conversion \\ \hline
	$\pi_{ij}$ & $\bar{\pi}_{ij}$ & $\bar{\pi}_{ij} = \gamma \pi_{ij}$ \\
	$c_{ij}$ & $\bar{c}_{ij}$ & $\bar{c}_{ij} = c_{ij}$ \\	
	$\mu^-_{i}$ & $\bar{\mu}_{i}^-$ & $\bar{\mu}^-_{i} = N \mu^-_{i}$ \\		
	$\mu^+_{j}$ & $\bar{\mu}_{j}^+$ & $\bar{\mu}^+_{j} = N \mu^+_{j}$ \\			
	$\alpha_{i}$ & $\bar{\alpha}_{i}$ & $\bar{\alpha}_{i} = \alpha_{i}$ \\		
	$\beta_{j}$ & $\bar{\beta}_{j}$ & $\bar{\beta}_{j} = \beta_{j}$ \\				
	$J$ & $\bar{J}$ & $\bar{J} = J N^2 \gamma$ \\					
	$\tau$ & $\bar{\tau}$ & $\bar{\tau} = \tau N$ \\ \hline
\end{tabular}
\end{table}
Finally, we make a note on the calculation of the coefficients $c_{ij}$ for the cost function $c(x,y) \coloneqq (x - y)^2$:
\begin{align*}
	c_{ij} &= N^2 \int_{\frac{i}{N}}^{\frac{i+1}{N}} \int_{\frac{j}{N}}^{\frac{j+1}{N}} (x - y)^2\,\dd y \dd x = ... = \frac{1}{N^2}\left((i - j)^2 + \frac{1}{6}  \right).
\end{align*}

\bigskip
\noindent

\textbf{Conflict of Interest:} The authors declare that they have no conflict of interest.

\bibliographystyle{plain}
\bibliography{refs}

\begin{thebibliography}{10}

\bibitem{ambrosio2013user}
Luigi Ambrosio and Nicola Gigli.
\newblock A user’s guide to optimal transport.
\newblock In {\em Modelling and optimisation of flows on networks}, pages
  1--155. Springer, 2013.

\bibitem{bertsekas2016nlp}
Dimitri~P. Bertsekas.
\newblock {\em Nonlinear programming}.
\newblock Athena Scientific Optimization and Computation Series. Athena
  Scientific, Belmont, MA, third edition, 2016.

\bibitem{blondel2018smoothOT}
Mathieu Blondel, Vivien Seguy, and Antoine Rolet.
\newblock Smooth and sparse optimal transport.
\newblock In Amos Storkey and Fernando Perez-Cruz, editors, {\em Proceedings of
  the Twenty-First International Conference on Artificial Intelligence and
  Statistics}, volume~84 of {\em Proceedings of Machine Learning Research},
  pages 880--889, Playa Blanca, Lanzarote, Canary Islands, 09--11 Apr 2018.
  PMLR.

\bibitem{carlier2017convergence}
Guillaume Carlier, Vincent Duval, Gabriel Peyr{\'e}, and Bernhard Schmitzer.
\newblock Convergence of entropic schemes for optimal transport and gradient
  flows.
\newblock {\em SIAM Journal on Mathematical Analysis}, 49(2):1385--1418, 2017.

\bibitem{chen1997}
Xiaojun Chen.
\newblock Superlinear convergence of smoothing quasi-newton methods for
  nonsmooth equations.
\newblock {\em Journal of Computational and Applied Mathematics}, 80(1):105 --
  126, 1997.

\bibitem{chen2002convergence}
Xiaojun Chen.
\newblock On convergence of {SOR} methods for nonsmooth equations.
\newblock {\em Numer. Linear Algebra Appl.}, 9(1):81--92, 2002.

\bibitem{chen2000semismooth}
Xiaojun Chen, Zuhair Nashed, and Liqun Qi.
\newblock Smoothing methods and semismooth methods for nondifferentiable
  operator equations.
\newblock {\em SIAM J. Numer. Anal.}, 38(4):1200--1216, 2000.

\bibitem{clason2019entropic}
Christian Clason, Dirk~A Lorenz, Hinrich Mahler, and Benedikt Wirth.
\newblock Entropic regularization of continuous optimal transport problems.
\newblock {\em arXiv preprint arXiv:1906.01333}, 2019.

\bibitem{cuturi2013sinkhorn}
Marco Cuturi.
\newblock Sinkhorn distances: {L}ightspeed computation of optimal transport.
\newblock In {\em Advances in neural information processing systems}, pages
  2292--2300, 2013.

\bibitem{cuturi2016smoothed}
Marco Cuturi and Gabriel Peyr\'{e}.
\newblock A smoothed dual approach for variational {W}asserstein problems.
\newblock {\em SIAM J. Imaging Sci.}, 9(1):320--343, 2016.

\bibitem{essid2018quadratically}
Montacer Essid and Justin Solomon.
\newblock Quadratically regularized optimal transport on graphs.
\newblock {\em SIAM Journal on Scientific Computing}, 40(4):A1961--A1986, 2018.

\bibitem{fonseca2007calculusofvariations}
Irene Fonseca and Giovanni Leoni.
\newblock {\em Modern methods in the calculus of variations: {$L^p$} spaces}.
\newblock Springer Monographs in Mathematics. Springer, New York, 2007.

\bibitem{genevay2018learning}
Aude Genevay, Gabriel Peyre, and Marco Cuturi.
\newblock Learning generative models with sinkhorn divergences.
\newblock In {\em International Conference on Artificial Intelligence and
  Statistics}, pages 1608--1617, 2018.

\bibitem{OPTpdecon}
Michael Hinze, Ren{\'e} Pinnau, Michael Ulbrich, and Stefan Ulbrich.
\newblock {\em Optimization with PDE constraints}, volume~23.
\newblock Springer Science \& Business Media, 2008.

\bibitem{kantorovich1942translocation}
Leonid~V. Kantorovi{\v{c}}.
\newblock On the translocation of masses.
\newblock {\em C. R. (Doklady) Acad. Sci. URSS (N.S.)}, 37:199--201, 1942.

\bibitem{papadakis2014proximal}
Nicolas Papadakis, Gabriel Peyr\'{e}, and Edouard Oudet.
\newblock Optimal transport with proximal splitting.
\newblock {\em SIAM J. Imaging Sci.}, 7(1):212--238, 2014.

\bibitem{peyre2019computational}
Gabriel Peyr{\'e} and Marco Cuturi.
\newblock Computational optimal transport.
\newblock {\em Foundations and Trends® in Machine Learning}, 11(5-6):355--607,
  2019.

\bibitem{rachev1998massI}
Svetlozar~T. Rachev and Ludger R\"{u}schendorf.
\newblock {\em Mass transportation problems. {V}ol. {I}}.
\newblock Probability and its Applications (New York). Springer-Verlag, New
  York, 1998.
\newblock Theory.

\bibitem{rachev1998massII}
Svetlozar~T. Rachev and Ludger R\"{u}schendorf.
\newblock {\em Mass transportation problems. {V}ol. {II}}.
\newblock Probability and its Applications (New York). Springer-Verlag, New
  York, 1998.
\newblock Applications.

\bibitem{roberts2017gini}
Lucas Roberts, Leo Razoumov, Lin Su, and Yuyang Wang.
\newblock Gini-regularized optimal transport with an application to
  spatio-temporal forecasting.
\newblock arXiv preprint arXiv:1712.02512, 2017.

\bibitem{santambrogio2015optimal}
Filippo Santambrogio.
\newblock {\em Optimal transport for applied mathematicians}, volume~87 of {\em
  Progress in Nonlinear Differential Equations and their Applications}.
\newblock Birkh\"{a}user/Springer, Cham, 2015.
\newblock Calculus of variations, PDEs, and modeling.

\bibitem{troe}
Fredi Tr{\"o}ltzsch.
\newblock Regular {L}agrange multipliers for control problems with mixed
  pointwise control-state constraints.
\newblock {\em SIAM Journal on Optimization}, 15:616--634, 2005.

\bibitem{villani2003topics}
C\'{e}dric Villani.
\newblock {\em Topics in optimal transportation}, volume~58 of {\em Graduate
  Studies in Mathematics}.
\newblock American Mathematical Society, Providence, RI, 2003.

\bibitem{villani2008optimal}
C\'{e}dric Villani.
\newblock {\em Optimal transport. Old and new}, volume 338 of {\em Grundlehren
  der Mathematischen Wissenschaften [Fundamental Principles of Mathematical
  Sciences]}.
\newblock Springer-Verlag, Berlin, 2009.

\bibitem{wright2015coordinatedescent}
Stephen~J. Wright.
\newblock Coordinate descent algorithms.
\newblock {\em Math. Program.}, 151(1, Ser. B):3--34, 2015.

\end{thebibliography}

\end{document}